\documentclass{article}

\usepackage[T1]{fontenc}
\usepackage[utf8]{inputenc}
\usepackage[english]{babel}
\usepackage{lmodern}
\usepackage[margin=3cm]{geometry}
\usepackage[backend=biber,style=alphabetic,sorting=nty,doi=false,isbn=false,url=false,minalphanames=2,maxbibnames=4,maxcitenames=4,sortcites]{biblatex}
\AtEveryBibitem{\clearlist{language}}

\usepackage{empheq}

\usepackage{authblk}
\usepackage[hidelinks]{hyperref}
\usepackage{enumitem}
\usepackage{tikz-cd} 	
\usepackage{mathtools}
\usepackage{alltt}
\usepackage{amsfonts}
\usepackage{amsmath}
\usepackage{amssymb}
\usepackage{amsthm}
\usepackage{tikz-cd} 	
\usepackage[nottoc]{tocbibind}
\usepackage{graphicx}
\usepackage[font={small,it}]{caption}
\usepackage{subcaption}
\usepackage{aliascnt}
\usepackage{cases}
\usepackage{comment}	
\usepackage[capitalize,nameinlink]{cleveref}
\usepackage{tikz}
\usetikzlibrary{tikzmark}
\usepackage{slashed}

\crefname{figure}{Fig.}{Fig.}
\Crefname{figure}{Fig.}{Fig.}
\crefname{subfigure}{Fig.}{Fig.}
\Crefname{subfigure}{Fig.}{Fig.}

\usepackage{xargs}   
\usepackage[colorinlistoftodos,prependcaption,textsize=tiny]{todonotes}
\newcommandx{\typo}[2][1=]{\todo[linecolor=red,backgroundcolor=red!25,bordercolor=red,#1]{#2}}
\newcommandx{\change}[2][1=]{\todo[linecolor=blue,backgroundcolor=blue!25,bordercolor=blue,#1]{#2}}
\newcommandx{\answer}[1]{\todo[linecolor=pink,backgroundcolor=pink!25,bordercolor=pink]{#1}}
\newcommandx{\unsure}[2][1=]{\todo[linecolor=green,backgroundcolor=green!25,bordercolor=green,#1]{#2}}
\newcommandx{\improve}[2][1=]{\todo[linecolor=violet,backgroundcolor=violet!25,bordercolor=violet,#1]{#2}}
\newcommandx{\thiswillnotshow}[2][1=]{\todo[disable,#1]{#2}}

\usepackage{tocloft}
\crefformat{equation}{(#2#1#3)}
\numberwithin{equation}{section}

\theoremstyle{definition}

\newtheorem*{notation}{Notation}

\newtheorem*{theorem*}{Theorem}
\newtheorem*{conjecture*}{Conjecture}

\theoremstyle{plain}
\newtheorem{theorem}{Theorem}[section]

\newtheorem{lemma}{Lemma}[section]
\newtheorem{prop}{Proposition}[section]
\newtheorem{conjecture}{Conjecture}[section]
\newtheorem{cor}{Corollary}[section]

\newtheorem{definition}{Definition}[section]

\theoremstyle{remark}
\newtheorem{remark}{Remark}[section]

\crefname{lemma}{Lemma}{Lemmata}
\crefname{prop}{Proposition}{Proposition}
\crefname{conjecture}{Conjecture}{Conjecture}
\crefname{cor}{Corollary}{Corollary}
\crefname{remark}{Remark}{Remark}
\crefname{defi}{Definition}{Definition}
\crefname{equation}{}{}
\crefname{enumi}{}{}
\crefname{appendix}{Appendix}{Appendices}

\newcommand{\dd}{\mathop{}\!\mathrm{d}}

\newenvironment{nalign}{
	\begin{equation}
		\begin{aligned}
		}{
		\end{aligned}
	\end{equation}
	\ignorespacesafterend
}

\renewcommand{\sl}{\mathfrak{sl}}

\newcommand{\C}{\mathcal{C}}

\newcommand{\N}{\mathbb{N}}

\newcommand{\R}{\mathbb{R}}

\newcommand{\T}{\mathbb{T}}

\newcommand{\scrip}{\mathcal{I}^+}
\newcommand{\scrim}{\mathcal{I}^-}

\newcommand{\abs}[1]{\left\lvert #1\right\rvert}

\newcommand{\norm}[1]{\left\lVert #1\right\rVert}

\newcommand{\floor}[1]{\lfloor #1 \rfloor}

\newcommand{\D}{\mathcal{D}}

\DeclareMathOperator{\supp}{supp}


\newcommand{\CH}{\mathcal{CH}}

\title{A note on exterior stability of isolated singularity formation\\ for nonlinear wave equations}

\author[1]{Istvan Kadar\thanks{istvan.kadar@math.ethz.ch}}
\author[2,3]{Lionor  Kehrberger\thanks{kehrberger@mis.mpg.de}}
\affil[1]{Department of Mathematics, ETH Zürich, Rämistrasse 101, Zürich}
\affil[2]{Max Planck Institute for Mathematics in the Sciences,  Inselstraße 22, 04103 Leipzig, Germany}
\affil[3]{University of Leipzig, Institute for Theoretical Physics, Br\"uderstraße 16,  04103 Leipzig, Germany}

\usepackage[makeroom]{cancel}
\usepackage{csquotes}
\newcommand{\brho}{\boldsymbol{\rho}}
\newcommand{\Cbar}{\underline{\C}}
\renewcommand{\O}{\mathcal{O}}
\newcommand{\pu}{\partial_u}
\newcommand{\pv}{\partial_v}
\renewcommand{\sl}{\slashed{\nabla}}
\renewcommand{\b}{\mathrm{b}}
\newcommand{\Vb}{\mathcal{V}_{\b}}
\newcommand{\Ve}{\mathcal{V}_{\mathrm{e}}}
\newcommand{\Hb}{H_{\b}}

\newcommand{\A}{\mathcal{A}}
\newcommand{\Aext}{\bar{\mathcal{A}}}
\newcommand{\Vbs}{\mathcal{V}_{\b,\scrip}}

\graphicspath{{figures}}

\begin{document}
	\maketitle
	\begin{abstract}
		We study the stability of the exterior of Type I and Type II singularity formation for the wave maps equation in $\R^{d+1}$ with $d\geq2$ and the power nonlinear wave equation in $\R^{d+1}$ with $d\geq3$:
		Given characteristic initial data on the backwards lightcone of the singularity $\C=\{t+r=0\}$ converging to the singular background solution along with suitable  data on an outgoing cone, we establish existence in a region $\{t+r\in(0,v_1),t-r\in(-1,0)\}$ for some suitably small $v_1$, i.e.~all the way to the Cauchy horizon. 

        Our result hinges on a particular set of assumptions on the regularity properties of these initial data, which conjecturally can be recovered by a more detailed stability analysis of the behaviour inside the past light cone; indeed, in certain settings, this was achieved in \cite{biernat_hyperboloidal_2021,kadar_smooth_2026}, and we strongly expect they can be proved in many other settings as well.

		The proof goes via a suitable change of coordinates and an application of the scattering result of \cite{kadar_scattering_2025}, which, in particular, also applies to scaling-critical potentials.
        
        While no symmetry assumption is made for the power nonlinear wave equation, we only provide the proof in the corotational symmetry class for the wave maps equation, but we also  sketch how to lift this restriction.
     
	\end{abstract}
	
	\setcounter{tocdepth}{2}
	\tableofcontents

\section{Introduction}\label{sec:intro}

\subsection{General Context}
	We study the power nonlinear wave equation and the wave maps equation in $\R^{d+1}$:
	\begin{subequations}
		\begin{align}
			P_{\mathrm{NW}}[\phi]&=\Box\phi+\phi^{p}=0,& \phi:\R^{d+1}\to\R\label{in:eq:nonlin}\tag{NW}, \quad d\geq 3;\\
			P_{\mathrm{WM}}[\vec{\phi}]&=\Box\phi^a+\phi^a(\partial\phi^b\cdot\partial\phi^b)=0,& \vec{\phi}:\R^{d+1}\to S^d,\quad d\geq 2\label{in:eq:wave_map}\tag{WM},
		\end{align}
	\end{subequations}
	where  $\Box=-\partial_t^2+\Delta_{\R^d}$, $\phi^a$ denote the components of $\vec{\phi}$ and we use Cartesian coordinates $x_i$, $i=1,\dots,d$ on $\R^d$. We also denote $r^2=\sum_i x_i^2$, and use the double null coordinates $u=\frac{t-r}{2}$, $v=\frac{t+r}{2}$.

    It is well known that both \cref{in:eq:nonlin} and \cref{in:eq:wave_map} can develop \emph{highly symmetric} singularities in finite time by concentrating a finite nonzero amount of scale critical norm at a single point \cite{shatah_weak_1988,rodnianski_formation_2010,krieger_renormalization_2008,krieger_slow_2009,raphael_stable_2012,hillairet_smooth_2012}.
    In the case of \cref{in:eq:wave_map}, the prominent symmetry setting is that of $K$-corotational wave-maps, where \cref{in:eq:wave_map} reduces to the scalar equation (see \cite{cazenave_harmonic_1998,raphael_stable_2012,biernat_non-self-similar_2015} for a derivation\footnote{For $K=1$, this symmetry assumption means that $\vec\phi$ can be written as,  for $\phi$ depending only on $|x|=r$ and $t$: 
    \begin{equation}
		\vec{\phi}=\begin{pmatrix} \frac{x}{\abs{x}}\sin(\phi(t,x)) \\ \cos(\phi(t,x)) \end{pmatrix}.
	\end{equation}}) taking the form
    \begin{equation}
\label{in:eq:wave_map_k}\tag{WM-s}
\begin{aligned}
P_{\mathrm{WM-s}}[\phi]
  &= \Box\phi + \frac{K(d+K-2)\sin(2\phi)}{2r^2}
  && \text{for }  d \ge 2,\, K\geq1; 
\end{aligned}
\end{equation}

In general, for equations such as \cref{in:eq:nonlin} and \cref{in:eq:wave_map}, the future boundary $\partial^+\mathcal{D}$ of the maximal globally hyperbolic development---see \cite{alinhac_blowup_1995} for a precise definition---may have a quite complicated structure allowing for vastly different phenomena, see  \cite{cote_construction_2013} for nonlinear wave equations in one space dimension and the introduction of \cite{kadar_scattering_2026} for a more extended literature review. 
In particular, the singular boundary may in general have spacelike segments; see, for instance, \cite{donninger_nonlinear_2010}.\footnote{For more exotic singularities, we refer the reader to \cite{bizon_self-similar_2010}.}
In this note, however, we show that for a large class of {isolated} singular solutions, \textbf{the picture of a future (null) Cauchy horizon emanating from the singularity is stable }(see the left of \cref{fig:exterior}).\footnote{In general, we call points $p\in\partial^+\mathcal{D}$ an isolated singularity if, for all sufficiently small annular neighbourhoods $U_d=\{x\in\Sigma: \abs{x-p}\in(d,2d)\}$ around $p$, $\Sigma\cap U_d$ is a Cauchy horizon, i.e.~the solution is extendible across $\Sigma$ in $U_d$.} See also \cite{biernat_hyperboloidal_2021,donninger_stable_2025} for related results.

One insight of this paper is that we can address the question of the nature of the singular boundary by splitting up the analysis into an \textit{interior} stability analysis, i.e.~a stability analysis inside the past light cone $\C$ emanating from the singularity, and an \textit{exterior} stability analysis to the future of the past light cone.
What this paper achieves is then the following: Based on known results in the literature on interior stability (which we adorn with some strengthening assumptions that in some cases are still conjectural), we show that the exterior stability analysis can be performed in a unified way using the scattering framework of our previous \cite{kadar_scattering_2025}.

	\subsection{Formulation of the problem}
    More precisely, we study the stability of such singular solutions in the exterior of the backwards lightcone emanating from the singularity in both the setting of self-similar (Type I) singularity formation, and faster than self-similar (Type II) singularity formation. 
   In all cases, we consider perturbations of \textit{isolated singularities}\footnote{In particular, we do not study spacelike singularities such as the ODE blow-up for \cref{in:eq:nonlin} discussed in \cite{merle_determination_2003,donninger_stable_2014}.} (although in some cases we can only infer the isolatedness of the singularities a posteriori from our exterior stability result, see already \cref{ref:rem:remmidemmi}). 

	For Type I solutions, we consider self-similar blow-up profiles $\phi_0$ (solving either \cref{in:eq:nonlin} or \cref{in:eq:wave_map_k}) of the form
	\begin{equation}\label{in:eq:self_similar}
		{\phi_0(t,x)=\big(2r-t\big)^{-c}\Phi_0\big(\frac{r}{2r-t}\big)}
	\end{equation}
	for some $\Phi_0(y)$ smooth in $\{\abs{y}\leq 3/2\}$ and $c$ a suitable constant.
All the examples of Type I singularities in this paper will either have $\Phi_0$  i) spherically symmetric for \cref{in:eq:nonlin} or ii) $K$-corotational for \cref{in:eq:wave_map_k}.
   At the same time, our results concerning the exterior stability of these symmetric solutions hold outside of these symmetry classes; the symmetry is only a feature of the given background $\Phi_0$.
   
    For energy-critical Type II solutions, the exterior background is simply $\phi_0=0$ for both \cref{in:eq:nonlin} and \cref{in:eq:wave_map_k} (for the latter, $\phi_0=0$ corresponds to $\vec{\phi}_0$ a point on $S^d$ for \cref{in:eq:wave_map}), and it is not necessary to specify further requirements.
   In other words, for energy-critical Type II solutions, our exterior analysis is perturbative around the trivial solution. 
   
   For the description of singular energy-supercritical Type II solutions, we refer the reader the paragraph around \cref{in:conj:type2_sup}. These are the solutions where the singularities are not a priori known to be isolated.

	The problem we study in this paper is formulated as follows: 
    Let $\C=\{v=0, u\geq -1\}$ denote the backwards light-cone of the singularity, and let $\Cbar=\{v>0, u=-1\}$ denote a transversal light cone; see \cref{fig:exterior}. 
	We then consider perturbations of $\phi_0$ by prescribing  initial data on $\C\cup\Cbar$ and solving the nonlinear problem
	\begin{equation}\label{in:eq:characteristic}
		P[\phi]=0, \qquad \phi|_{\C}=\phi_{\C}=\bar{\phi}_{\C}+{\phi}_0|_{\C}, \, \phi|_{\Cbar}=\phi_{\Cbar}=\bar{\phi}_{\Cbar}+\phi_0|_{\Cbar},
	\end{equation}
    where $P$ denotes either \cref{in:eq:nonlin} or \cref{in:eq:wave_map_k}, and $\phi_0$ is as above.

    The main result of the paper is that under suitable initial data assumptions on the perturbations $\phi_{\C}$ and $\phi_{\Cbar}$, and for $v_1>0$ sufficiently small, the solution to \cref{in:eq:characteristic} exists in the domain $\D_{-1,v_1}=\{v\in(0,v_1), u\in[-1,0)\}$ and, in fact, can be {({non-uniquely})} extended across the Cauchy horizon $\CH=\{v>0, u=0\}$. 
   If we pose the data at $u=-\epsilon$ for $\epsilon\ll1$ and make additional assumptions, we can also show that $v_1$ can be taken to be infinite.
    The ``suitable'' initial data assumptions have been recovered in certain interior stability analyses. We strongly expect them to be recoverable in many other settings; we give more detail below.
    In particular, we expect that following further progress on interior stability analysis, our exterior stability result will eventually be applicable to any compactly supported smooth perturbations at, say, $t=-1$, of all singular solutions considered in this paper.
    This paper being entirely about  exterior stability analysis, we simply state our data assumptions as conjectured interior stability results.
    	
	\begin{figure}[htbp]
		\centering
		\includegraphics[width=0.7\textwidth]{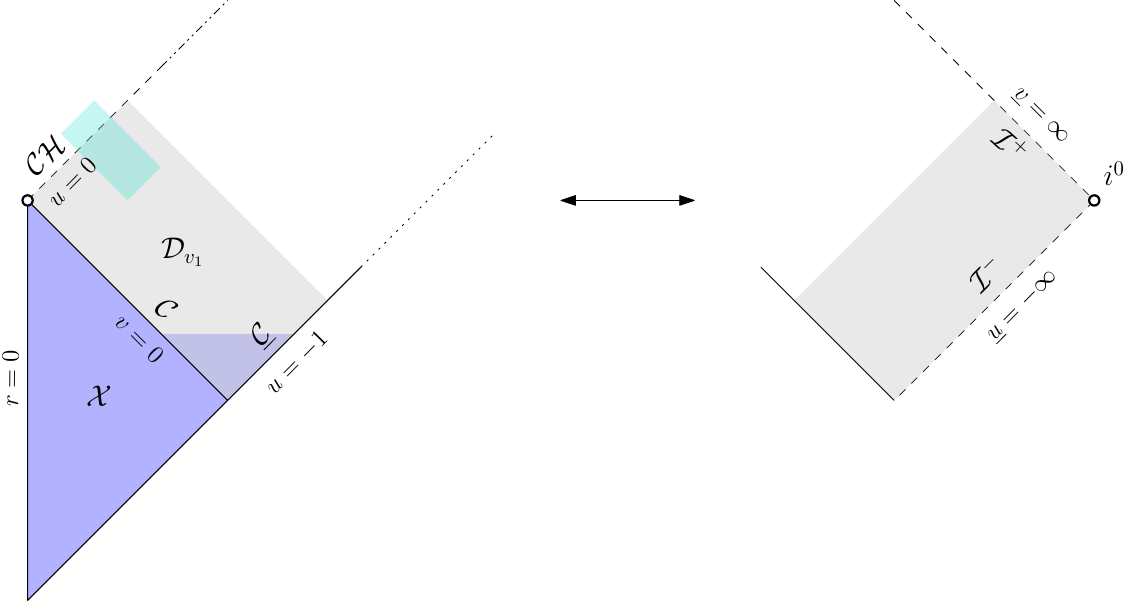}
		\caption{On the left: We assume the existence of a solution in $\mathcal{X}=(\{v\leq 0\}\cup \{t\leq -0.9\})\cap\{-1\leq u<0\}$ forming a singularity at $u=0=v$. We then show that under suitable assumptions, the solution can be extended to all of $\mathcal{D}_{v_1}$ (shaded in grey) and, in fact, can be extended beyond the Cauchy horizon. 
        The crucial observation of this paper is that a coordinate transformation transforms this problem into a scattering problem from $\scrim$ to $\scrip$ on the Minkowski spacetime, depicted on the right.}
		\label{fig:exterior}
	\end{figure}
    

	
	The remainder of this introduction is structured as follows: We recall some of the literature on the study of the interior problem in \cref{in:sec:interior}. In particular, we motivate the assumptions that we place on the initial data for \cref{in:eq:characteristic}.
   In \cref{in:sec:link}, we then explain the coordinate transformation that connects the results of \cite{kadar_scattering_2025} to the problem \cref{in:eq:characteristic}. We state our main result in \cref{in:sec:main}.

	\subsection{Description of singular solutions in the interior solutions}\label{in:sec:interior}
	Finite time singularity formation for \cref{in:eq:nonlin} and \cref{in:eq:wave_map} being an extremely large subject, we only focus on those results directly relevant to our main result; see \cite{raphael_quantized_2014,duyckaerts_soliton_2023} for more background and \cite{biernat_non-self-similar_2015} for a comprehensive heuristic discussion.
	In particular, we only concentrate on \emph{point}-like singularity formation, see \cite{merle_determination_2003,merle_openness_2008,bizon_self-similar_2010,merle_existence_2012,cote_construction_2013,donninger_nonlinear_2010,kichenassamy_fuchsian_2007,donninger_blowup_2016} for other types of singularities.
	
	\textbf{We begin with Type I, i.e. self-similar, singularity formation:}
    The solutions in this part are of the form \cref{in:eq:self_similar}, where $c$ is chosen such that the nonlinearity and $\Box\phi$ in \cref{in:eq:nonlin,in:eq:wave_map} have the same scaling behaviour towards $\{r=t=0\}$.
    Therefore, in the exterior region, our solutions will take the form $\phi=\phi_0+\bar\phi$ for some perturbation $\bar\phi$, where the existence of $\phi_0$ is provided by:
    
	\begin{theorem}[Type I blow-up for  \cref{in:eq:wave_map} \cite{shatah_weak_1988,biernat_generic_2015} and \cref{in:eq:nonlin} \cite{glogic_co-dimension_2021}]\label{in:thm:self_similar_existence}
    ~
        \begin{enumerate}[label=(\alph*),start=1]
            \item For any $d\geq 3$ and $K=1$, there exist solutions $\phi_0$ to \cref{in:eq:wave_map_k} of the form \cref{in:eq:self_similar} with $\Phi_0$ smooth on $[0,3/2]$ and $c=0$.
            \item For any $d\geq 5$ and for $p=3$,\footnote{We expect the same to hold also for $d=3$ and $(p-7)\in2\N$ based on \cite{bizon_self-similar_2007}, however \cite{bizon_self-similar_2007} does not discuss the behaviour of the self-similar solution at the forward lightcone $\{r/t=1\}$, so further analysis would be required. Note that not all self similar solutions are smooth up to the Cauchy horizon, see \cite{bizon_self-similar_2010}.} there exist solutions $\phi_0$ to \cref{in:eq:nonlin} of the form \cref{in:eq:self_similar} with $\Phi_0$ smooth on $[0,1]$ and $c=1$.
        \end{enumerate}
		
	\end{theorem}
	There are even explicit formulae for the $\Phi_0$ as given in \cite{biernat_generic_2015}.
	Beyond their existence, (some of) these exactly self-similar solutions have been shown to be \textit{stable} inside the past lightcone:
	\begin{theorem}[Stability of Type I blow-up solutions \cite{donninger_stable_2011,costin_proof_2016,costin_mode_2017,biernat_hyperboloidal_2021,glogic_co-dimension_2021}]\label{in:thm:self_similar_stability}\label{thm:12}
		The solutions of \cref{in:thm:self_similar_existence} a) are stable in corotational symmetry\footnote{See \cite{weissenbacher_mode_2025} for a linear stability result outside symmetry.} with respect to smooth perturbations inside the backward lightcone, up to time translation of the point of blow up.

        b) The solutions of \cref{in:thm:self_similar_existence} b) for $d=7$ are co-dimension one stable with respect to smooth perturbations in the backward lightcone, up to spacetime translation of the point of blow up.
	\end{theorem}
    
	While \cref{in:thm:self_similar_stability} comes with certain quantitative control, we actually require more precise control in order to apply our main result; in particular, we require regularity with respect to the vector fields $u\partial_u, x_i\partial_{x_j}-x_j\partial_{x_i}$ (i.e.~conormal regularity \textit{along $\C$}), and, moreover, with respect to the vector field $r\partial_t$ (which corresponds to conormal regularity \textit{in the interior} of $\C$; see already \cref{rem:in:relevance}).\footnote{Here, and throughout the rest of the paper, a $\pu$ or $\pv$ derivative is always meant to be evaluated in double null coordinates, whereas $\partial_t$ and $\partial_{x_j}$-derivatives are always meant to be evaluated in Cartesian coordinates, and $\partial_r$ is meant to be evaluated in $(t,r, x/|x|)$-coordinates.} 
We state this regularity as conjecture for the perturbed solutions $\phi$, assuming they are all renormalised to a time of singularity formation $t=0$:
	\begin{conjecture}[Type I: Stability with additional conjectured regularity]\label{in:conj:type1}
    	All existence statements below refer to existence in the region $\mathcal{X}=(\{v\leq 0\}\cup \{t\leq -0.9\})\cap\{-1\leq u<0\}$.
        Let  $\phi$ denote a solution of \cref{in:thm:self_similar_stability} arising from smooth initial data.

    There exists $\epsilon>0$ (depending on the equation), such that for any $k\in\N$ the solutions satisfy
		\begin{enumerate}[label=(\alph*)]
			\item \label{type1}  for \cref{in:eq:wave_map_k} with $d\geq3$ as in \cref{in:thm:self_similar_stability} (a) for any $0\leq |\alpha|\leq k$:\begin{equation}\label{in:eq:con_rates}
				\left|\{u\partial_u, x_i\partial_{x_j}-x_j\partial_{x_i},r\partial_t,1\}^\alpha(\phi-\phi_0)\right||_{\C}\lesssim_{|\alpha|} r^{\epsilon};
            \end{equation}
			 \item for \cref{in:eq:nonlin} with $d=7$ and $p=3$ as in \cref{in:thm:self_similar_stability} (b) for any $0\leq |\alpha|\leq k$\label{type1:wave}\begin{equation}\label{in:eq:con_rates2}
				\left|\{u\partial_u, x_i\partial_{x_j}-x_j\partial_{x_i},r\partial_t,1\}^\alpha(\phi-\phi_0)\right||_{\C}\lesssim_{|\alpha|} r^{\epsilon-1}.
			\end{equation}
		\end{enumerate}
\end{conjecture}
 We refer to $\phi$ as above as a $k$-regular solution. 
    In the sequel, we will write $(\phi-\phi_0)|_{\C}=\mathcal{O}_k(r^{\epsilon})$ for functions defined in a neighbourhood of $\C$ satisfying $\{u\partial_u, x_i\partial_{x_j}-x_j\partial_{x_i},r\partial_t,1\}^\alpha(\phi-\phi_0)|_{\C}\lesssim r^\epsilon$ (or simply $\mathcal{O}(r^{\epsilon})$ if $k=\infty$) rather than \cref{in:eq:con_rates}.
    \begin{remark}[The relevance of $r\partial_t$-regularity.]\label{rem:in:relevance}
    Note that the $u\pu$ and $x_i\partial_{x_j}-x_j\partial_{x_i}$ are tangential derivatives along $\C$ measuring \textit{conormal} regularity. The $r\partial_t$ regularity, on the other hand, cannot be formulated as an assumption solely on the restriction of $\phi$ to $\C$; it should instead be viewed as an assumption coming from the behaviour of $\phi$ in the interior. 
    Consider, for instance, the wave maps equation \cref{in:eq:wave_map_k} with $K=1$: 
		\begin{equation}
			\left(-\partial_t^2+\partial_r^2+\frac{d-1}{r}\partial_r\right)\phi-\frac{d-1}{2}\frac{\sin(2\phi)}{r^2}=\left(-\partial_u\partial_v+\frac{d-1}{2r}(\partial_v-\partial_u)\right)\phi-\frac{d-1}{2}\frac{\sin(2\phi)}{r^2}.
		\end{equation} 
		Provided that $\phi|_{\C}\sim r^\epsilon$, we can restrict the above equation to $\C$ and integrate from $\Cbar\cap \C$ for $\partial_v\phi$.
		We compute that in general $r^{\frac{d-1}{2}}\partial_v\phi\sim C+r^{\frac{d-1}{2}-1+\epsilon}$, for $C=\lim_{r\to 0}r^{\frac{d-1}{2}}\pv\phi$. One similarly obtains higher-order $\pv$-derivatives. Thus, the rates \cref{in:eq:con_rates} only hold if all the limits $\lim_{r\to\infty}r^{\frac{d-1}{2}}\pv^{n}\phi$ vanish for $n\leq k$.
		This condition can be linked to regularity in the interior of $\C$ as follows: 
        If we require $\phi$, along with its $t\partial_t$ and $r\partial_r$ derivatives, to satisfy uniform $r^\epsilon$ decay rate in the region $r/t\in[-1,-1/2]$, then this implies \cref{in:eq:con_rates} (see already  \cref{scat:reverse_peeling}).
	\end{remark}
 
  \begin{remark}[Known cases of \cref{in:conj:type1}]\label{rem:known1}
    The additional regularity assumptions \cref{in:eq:con_rates,in:eq:con_rates2} are not provided by \cref{thm:12}, see for instance the restriction $k\leq 3$ in \cite[Theorem 1.1]{glogic_co-dimension_2021}. 
    We nevertheless strongly expect them for structural reasons. 
    Indeed, already in
        \cite[Theorem 1.2]{biernat_hyperboloidal_2021}, which studies \cref{in:eq:wave_map_k} for $d=3$, we see that for any $k$, there exist solutions satisfying \cref{in:eq:con_rates} for any $|\alpha|\leq k$.\footnote{In fact, the main result of \cite{biernat_hyperboloidal_2021} does not quite show \cref{in:eq:con_rates} for all $k\geq1$ and thus does not give: smooth data implies infinite $\{t\partial_x,t\partial_t\}$ regularity. Instead, for any $m\in\N_{\geq8}$, they show that $H^m$ perturbations lead to \cref{in:eq:con_rates} with up to $k=m-3$. We expect that a slight modification of their argument can lead to the propagation of regularity producing solutions that satisfy \cref{in:eq:con_rates} for any $k$.} Since our main result \cref{in:thm:main} only requires finite regularity, it is therefore unconditional in this case.
    \end{remark}

	\paragraph{Next, we consider \textit{energy-critical} Type II singularity formation:}
    In contrast to the self-similar Type I singularities discussed before,  the singularities in this part have two different relevant length scales.
    On scales where $r\sim t^\nu$ for some $\nu>1$ (i.e.~inside the light cone $\C$), the solution takes the form $t^{c\nu}W(x/t^\nu)$, where the profile $W$ solves the stationary problem $P[W(x)]=0$, and $c$ is determined by scaling.
    In the region $r\sim t$, the scale critical norm goes to 0, that is, $\phi$ is approximated by $\phi_0=0$.
    Therefore, in the exterior region, our solutions will take the form $\phi=0+\bar\phi$ for some perturbation $\bar\phi$.
    
    The perturbations of \cref{in:conj:type1} are  smooth across $\C$ because the data for the background $\Phi_0$ and the perturbation is posed on a Cauchy hypersurface, and regularity is propagated inside the domain of existence.\footnote{We contrast such smooth self-similar solutions with the celebrated result of Christodoulou \cite{christodoulou_examples_1994}, where finite regularity Type I singularities have been constructed for the Einstein scalar field system.} 
    This is not always the case for Type II singularity formation.
	The reason is that for faster than self-similar concentrations, the two main methods to construct the solution yield qualitatively different features:
	The \emph{forward} approach---as pioneered by \cite{rodnianski_formation_2010,raphael_stable_2012}---yields solutions that are smooth across $\C$, while the \emph{backwards} approach---going back to \cite{krieger_renormalization_2008}---yields limited $C^\nu$-regularity across $\C$ for some finite $\nu\in\R$.
	In the latter work, the authors do not  explicitly keep track of the precise regularity, but an inspection of the proof yields that the solution is {\textit{conormal}} \textit{across} $\C$, i.e.~it is infinitely regular with respect to the vector fields $\{v\partial_v,x_i\partial_{x_j}-x_j\partial_{x_i}\}$, and moreover w.r.t.~$t\partial_u$.\footnote{An example of a conormal function along $\Cbar$ that is only of regularity $C^{\nu}$ with $\nu\notin\N$ is $c_0+c_1v+\ldots+c_{\floor{\nu}}v^{\floor{\nu}}+v^\nu$ for $c_i\in\R$.}
	
   Below, we quote various results on Type II singularity formation. As before,, we write them as conjectures because the cited results do not always include explicit decay rates nor mention conormality. In particular, only some of the works cited below include high $r\partial_x, r\partial_t$-regularity along $\C$:
    
	\begin{conjecture}[Type II singularity formation, energy-critical]\label{in:conj:type2}
		All existence statements below refer to existence in the region $\mathcal{X}=(\{v\leq 0\}\cup \{t\leq -0.9\})\cap\{-1\leq u<0\}$.
		\begin{enumerate}[label=(\alph*),start=3]
			\item \label{wm:d=2}
            For any $\nu>1$, there exists $\phi$ solving \cref{in:eq:wave_map_k} with $d=2, \, K=1$ s.t.~$\phi|_{\C}=\O\big(r^{\nu-1}\big)$ and s.t.~$\phi$ is $C^{\nu-1/2-}$ across $\C$ with infinite conormal regularity \cite{krieger_renormalization_2008,krieger_full_2014};
			\item \label{nw:d=3}For any $\nu>1$, there exists $\phi$ solving \cref{in:eq:nonlin} with $d=3$, $p=5$ s.t.~$\phi|_{\C}=\O(r^{\frac{\nu-2}{2}})$ and $\phi$ is $C^{\frac{\nu}{2}-}$ across $\C$ with infinite conormal regularity
			\cite{krieger_slow_2009,kadar_smooth_2026};
			\item \label{nw:d=4}
            For any $\nu>1$, there exists $\phi$ solving \cref{in:eq:nonlin} with $d=4$, $p=3$ s.t.~$\phi|_{\C}=\O(r^{\nu-2})$ and s.t.~$\phi$ is $C^{\frac{\nu}{2}-}$ across $\C$ with infinite conormal regularity \cite{samuelian_construction_2024};
			\item \label{wm:d=2_poly} 
            For any $\nu\in\N_{\geq 2}$, there exists
        $\phi$ solving \cref{in:eq:wave_map_k} with $d=2$ and $K=1$ s.t.~$\phi|_{\C}=\O(r^{\nu})$  and is smooth across  $\C$ \cite{jeong_quantized_2025};\footnote{We ignored extra logarithmic factors appearing, as the upper bounds suffices for the discussion here.}
			\item \label{wm:d=2_log}
            There exists $\phi$ solving \cref{in:eq:wave_map_k} with $d=2$ and $K=1$ s.t.~$\phi|_{\C}=\O\big(e^{-\sqrt{|\log(r)|}}\big)$ and s.t.~$\phi$ is smooth across $\C$ \cite{raphael_stable_2012,kim_sharp_2023};
            \item \label{wm:d=2_log_k}
            There exists $\phi$ solving \cref{in:eq:wave_map_k} with $d=2$ and $k\geq2$ s.t.~$\phi|_{\C}=\O\big(|\log(r)|^{\frac{-1}{2k-2}} \big)$ and s.t.~$\phi$ is smooth across $\C$ \cite{raphael_stable_2012}; 
			\item  \label{nw:d=4_log}
            There exists $\phi$ solving \cref{in:eq:nonlin} with $d=4$, $p=3$, $\phi|_{\C}=\O(r^{-1}e^{-\sqrt{|\log(r)|}+O(1)})$ and s.t.~$\phi$ is smooth across $\C$ \cite{hillairet_smooth_2012}.
		\end{enumerate}
		The solutions above are smooth in the interior of $\C$, i.e.~for $v<0$.
        Moreover, their scale critical norm on $\{t=-\tau,\abs{r}\leq \tau\}$ remains bounded away from 0 as $\tau\to0$.

        Finally, in all statements, smoothness can be replaced by $C^k$, infinite conormal regularity across $\C$ can be replaced by conormal regularity of degree $k$, and the notation $\mathcal{O}(\dots)$ introduced below \cref{in:conj:type1} can be replaced by $\mathcal{O}_{k}(\dots)$. We then refer to the conjectured solutions as "$k$-regular solutions".
	\end{conjecture}
    \begin{remark}[Explanation of the rates along $\C$]
   
In all cases above, the rates along $\C$ are obtained by assuming the leading-order rate to be the same as the tail of the soliton evaluated at the concentration speed.
    We explain this via the example of \cref{wm:d=2}:
		In \cite{krieger_renormalization_2008,krieger_full_2014}, the authors consider singularity formation where the solution of \cref{in:eq:wave_map_k} decomposes as
		\begin{equation}
			\phi=2\arctan (r/t^\nu)+\epsilon.
		\end{equation}
	Here $\epsilon$ is an error term that we ignore for the heuristic. We then use that $2\arctan(r/t^{\nu})|_{\C}=\pi+O(r^{\nu-1})$ to obtain \cref{wm:d=2}.
	\end{remark}

    \begin{remark}[Local well-posedness]\label{rem:local}
        Note that the reference of \cref{in:conj:type2} all show convergence in the energy topology $\dot{H}^1\times L^2$ on constant $t$ slices.
        Therefore, existence of the Cauchy horizon already follows from the local existence in the critical norm. In these cases, the contribution of this paper is then to provide a sharper description of the solution, see \cref{rem:sharP}.
    \end{remark}
    
    \begin{remark}[Known cases of \cref{in:conj:type2}]\label{rem:known2}
         As in \cref{rem:known1}, the specific $\{r\partial_x,r\partial_t\}$-regularity encoded in the $\O$-notation has not yet been proved in all cases. It has been proved verbatim in \cite{kadar_smooth_2026} in case \cref{nw:d=3}.
       There are further cases where the results available in the literature suffice for \cref{in:thm:main} to apply.
        For instance, consider \cref{wm:d=2}. In \cite{krieger_renormalization_2008}, the authors write the solution $\phi=\phi_N+\epsilon_\alpha$, where they construct $\phi_N$ explicitly infinitely regular with respect to $\{v\partial_v,u\partial_u\}$  and $\epsilon$ satisfying the bound
        $\norm{\epsilon_\alpha}_{H^\alpha}\lesssim t^{\nu(N-2)}$ for any $\abs{\alpha}\in(1/4,\nu/2)$.
        Moreover, following the symbolic bounds as in \cite[Proposition 4.7]{krieger_renormalization_2008}, we expect that $\epsilon$ already satisfies conormality across $\C$.
        Similar commentary applies to \cref{nw:d=4}.

        Therefore, in these three cases, our \cref{in:thm:main} applies unconditionally.
    \end{remark}

   \begin{remark}[Regularity across $\C$: Conormality vs Smoothness]
		Notice that while the solutions of \cref{wm:d=2_poly,wm:d=2_log,wm:d=2_log_k,nw:d=4_log,wm:d>7,nw:supercritical} as well as \cref{type1,type1:wave} all are smooth across $\C$, \cref{wm:d=2,nw:d=3,nw:d=4} are not. 
        With more detail given in the main body, let us briefly discuss the importance of the regularity across $\C$ through the example of \cref{nw:d=4}.
		For our main result \cref{in:thm:main}, we  need to show that $\phi\sim\mathcal{O}(r^{a_0})$, with $a_0>-1$, in the region $\{t/r\in(-1/2,1/2)\}$.
		This requires $(r\partial_t)$-regularity and a corresponding decay rate along $\C$ in order to construct an approximate solution near $\C$ and peel off sufficiently many terms along $\Cbar$.
        In turn, this requires that the regularity across $\Cbar$ is larger than $a_0+\frac{d-1}{2}$; we can only deduce this rate if the regularity across $\C$ is higher than $a_0+\frac{d-1}{2}$. 
        We will show that these conditions are always satisfied in all the cases above.
	\end{remark}
    
	\begin{remark}[Polynomial vs logarithmic perturbation]
		For the settings \cref{wm:d=2_log,wm:d=2_log_k,nw:d=4_log}, we cannot directly apply the results of \cite{kadar_scattering_2025}. 
        We nevertheless expect that the techniques of \cite{kadar_scattering_2025} can be slightly extended to also apply to these cases; see \cref{applications:rem} for further discussion.
	\end{remark}


    \paragraph{Finally, we consider \textit{energy-supercritical} Type II singularity formation:}
    
    As for the case of energy-critical Type II problems, there are two scales of the singularity formation, $r\sim t^\nu$ for some $\nu>1$ and $r\sim t$.
    In the region where $r\sim t^\nu$, $\phi$ is, as before, to leading order the stationary profile $t^{c\nu}W(x/t^\nu)$ solving $P[W(x)]=0$, while in the region with $r\sim t$,  the leading order behaviour is a non-trivial self-similar profile $\phi_0$ of the form \cref{in:eq:self_similar} with $\Phi_0$ \emph{singular at the origin}.
    This singular (at $r=0$) solution is $c_{p,d}r^{-\frac{2}{p-1}}$ with $\frac{p-1}{2}c_{p,d}^{p-1}=d-2-\frac{2}{p-1}$ for \cref{in:eq:nonlin}, and the equatorial map $\phi=\frac{\pi}{2}$ for \cref{in:eq:wave_map_k}.

    \begin{conjecture}[Type II singularity formation, energy-supercritical]\label{in:conj:type2_sup}
		All existence statements below refer to existence in the region $\mathcal{X}=(\{v\leq 0\}\cup \{t\leq -0.9\})\cap\{-1\leq u<0\}$.
		\begin{enumerate}[label=(\alph*),start=10]
			\item \label{wm:d>7}
            For $\gamma=\frac{1}{2}(d-2-\sqrt{d^2-8d+8})$ and any $l\in\N_{>\gamma}$ there exists
            $\phi$ solving \cref{in:eq:wave_map_k} with $K=1$, $d\geq7$, such that $\phi|_{\C}=\frac{\pi}{2}+\O(r^{l-\gamma})$ with $\phi$ smooth across $\C$ \cite{ghoul_construction_2018};
			\item \label{nw:supercritical} 
            For $d\geq11$ and $p>1+\frac{4}{d-4-2\sqrt{d-1}}$ fix $\alpha_{d,p}>0$ as in \cite[Theorem 1.1] {collot_type_2018} and $l\in\N_{>\alpha}$, then there exists $\phi$ solving
            \cref{in:eq:nonlin} such that $\phi|_{\C}=c_{p,d}r^{-\frac{2}{p-1}}+ \O(r^{l-\alpha-\frac{2}{p-1}})$ and smooth across $\C$ \cite{collot_type_2018}.
		\end{enumerate}
		The solutions above are smooth in the interior of $\C$, i.e.~for $v<0$.
        Moreover, their scale critical norm on $\{t=-\tau,\abs{r}\leq \tau\}$ is bounded from below as $\tau\to0$.

        Finally, in all statements, smoothness can be replaced by $C^k$ and the notation $\mathcal{O}(\dots)$ introduced below \cref{in:conj:type1} can be replaced by $\mathcal{O}_{k}(\dots)$. We then refer to the conjectured solutions as ``$k$-regular solutions''.
	\end{conjecture}
\begin{remark}[Known cases for \cref{in:conj:type2_sup}]\label{rem:known3}
In contrast to the previous two conjectures and \cref{rem:known1,rem:known2}, we are not aware of any results in the literature providing the required $\{r\partial_t\}$-regularity along $\C$ captured in the $\O$-notation. 
We nevertheless expect this regularity to hold for structural reasons.
\end{remark}
\begin{remark}
 \label{ref:rem:remmidemmi}   
 For the solutions of \cref{in:conj:type2_sup}, it is not known a priori that the singularities are indeed isolated. This should be contrasted to the energy-critical case, where this can be inferred from a local existence argument (see \cref{rem:local}). Therefore, it is only after our main result \cref{in:thm:main} that the singularities can be inferred to be isolated.
\end{remark}

    \subsection{Relation to the global scattering result \texorpdfstring{\cite{kadar_scattering_2025}}{[KK25]} }\label{in:sec:link}
    The main observation of this paper is that, given the conjectural assumptions of \cref{in:conj:type1,in:conj:type2,in:conj:type2_sup} as input for  \cref{in:eq:characteristic}, we can solve \cref{in:eq:characteristic} via an application of the semi-global scattering results of \cite{kadar_scattering_2025}.
    More precisely, by considering the change of coordinates
    \begin{equation}
        (\underline{u},\underline{v})=\Phi(u,v)=(-v^{-1},-u^{-1})
    \end{equation}
    mapping the origin ($t=0=r$) to spacelike infinity, $\C$ to past null infinity $\mathcal{I}^-$ and $\CH$ to future null infinity $\mathcal{I}^+$, we can recast \cref{in:eq:characteristic} as a semi-global scattering problem, for which existence and sharp decay and (conformal) regularity results have been treated in generality in \cite{kadar_scattering_2025}.
    See \cref{fig:exterior}.     For the interested reader, we included a brief sketch of the techniques of \cite{kadar_scattering_2025} in \cref{sec:app:sketch} tailored to the finite setting of $\D_{v_1}$.

    \begin{remark}[Long-range potentials and self-similar singularities]
        The proof of \cref{in:thm:main} follows by reducing \cref{in:eq:nonlin,in:eq:wave_map_k} to a global problem in $\{\underline{u}<-1,\underline{v}>1\}$.
        In the case of Type I and energy supercritical Type II singularities, we crucially use that  in \cite{kadar_scattering_2025}, we have already proved strong linear estimates for operators of the form
        \begin{equation}
            \Box+r^{-2}V(t/r)
        \end{equation}
        for arbitrary $V\in C^\infty$. Such control is obtained via the constant $\bar{c}$ appearing in the  computation~\cref{app:eq:divJpm}.
    \end{remark}

	\subsection{Main result: Exterior stability}\label{in:sec:main}
	We are now ready to state our main result:
	\begin{theorem}\label{in:thm:main}
		Let $\phi$ be a $k$-regular solution as in \cref{in:conj:type1,in:conj:type2,in:conj:type2_sup}, excluding \cref{wm:d=2_log_k}, and let $\phi_{\C}$ and $\phi_{\Cbar}$ denote the restrictions of $\phi$ to $\C$ and $\Cbar$, respectively. Moreover, in the case of \cref{in:eq:wave_map}, assume the solution to be $K$-corotational.
        For $k$ sufficiently large depending only on $d$, the characteristic initial value problem \cref{in:eq:characteristic} has a solution in all of $\D_{v_1}$ for $v_1$ sufficiently small, i.e.~the solution $\phi$ extends to $\D_{v_1}$, where the smallness of $v_1$ only depends on weighted norms of $\phi$ along $\C\cap\Cbar$ and on spacetime norms of the background solution $\phi_0$.

        Furthermore, in all cases, for $v_1/2<v<v_1$, the solution can be extended across the Cauchy horizon  as a weak solution. See \cref{fig:exterior}.
		
	\end{theorem}
    We emphasise that the only input into the theorem are the properties of $\phi$ given in the statements of the conjectures.
    The relevant function spaces are given in  \cref{scat:cor:main,sec:not}.
    The proof is given in \cref{sec:appandproof}.
\begin{remark}[Outside of corotational symmetry]
    We give a short sketch of how a proof outside of symmetry in the wave maps setting might proceed in  \cref{subsec:extendingadmissiblity}. (We make no symmetry assumption in the case \cref{in:eq:nonlin}.)
\end{remark}
\begin{remark}[Global Cauchy horizon]\label{rem:global}
    If we extend the interior solution of \cref{in:conj:type2} by posing data along $u=-\epsilon$ for $\epsilon\ll1$ instead of along $\Cbar$, we may also establish existence all the way up to future null infinity for energy-critical Type II singularities under suitable assumptions. See \cref{sec:3:scrip} and \cref{scat:cor:null_inf}.
\end{remark}
\begin{remark}[Sharp regularity of $\CH$]\label{rem:sharP}
   Beyond the existence of the Cauchy horizon, one may further study the precise regularity of the perturbations at $\CH$. While this is not the focus of this paper, the methods of \cite{kadar_scattering_2025} also allow to address this question. Indeed, in the case of \emph{critical} Type~II singularity formation, assuming the precise leading order behaviour $\phi_{\C}$, we can directly use the results of \cite{kadar_scattering_2025}; cf.~\cref{scat:lem:sharp} in \cref{sec:lower_bound} below which discusses the case of $d=3$, applicable to \cref{nw:d=3}. The case of other spatial dimensions (as well as the energy-supercritical case) can be dealt with similarly by using the ODE results for exact $c/r^2$-potentials of \cite[Section 10]{kadar_scattering_2025}. 

   On the other hand, for Type I singularities, one would need to perform a new ODE analysis similar to~\cite[Section 10]{kadar_scattering_2025}; these cases therefore \textit{do not} follow directly.%
\end{remark}
\begin{remark}[Extendibility of the solution]
    As is clear from \cref{rem:local}, the solutions can be extended across $\CH$ for Type II energy-critical problems for $t>0$.
    These extensions are, by a propagation of singularities argument, smooth in $\{t>0\}\setminus\CH$.
    Alternative extension with expanding solitons for $t>0$ are also  expected following \cite{jendrej_dynamics_2022}.
    Whether a selection criterion for a physically relevant extension can be obtained via an inviscid limit or a Skyrme type regularisation in the case of \cref{in:eq:wave_map}---see \cite{geba_large_2018,geba_large_2019} for global wellposedness in the latter---remains entirely open.
    For Type I and energy supercritical Type II solutions, even the extendibility has so far not been addressed, see however \cite{bizon_dispersion_2000,bonk_future_2026} for a related discussion. A selection criterion, too, remains elusive.
\end{remark}

\paragraph{Acknowledgement:}
The authors would like to thank Serban Cicortas for bringing the problem to their attention and the hospitality of the Simons Center for Geometry and Physics during the workshop on Partial Differential Equations of Classical Physics where this work was initiated.
The first author acknowledges the support of the SNSF starting grant TMSGI2 226018. The second author  acknowledges support through the ERC Starting Grant 101115568.

	\section{Norms and notation}\label{sec:not}
	
	In this section, we introduce some norms that are useful to understand solutions up to the Cauchy horizon in the region $\D_{u_1,v_1}:=\{v\in(0,v_1),u\in(u_1,0)\}$.
    We also introduce the notation $\Cbar_{u_1,v_1}=\{u=u_1, v\in(0,v_1)\}$, $\C_{u_1,v_1}=\{u\in(u_1,0),v=v_1\}$, and simply write $\Cbar_{u_1,\infty}=\Cbar_{u_1}$ as well as $\Cbar_{-1}=\Cbar$.
	We introduce
	\begin{equation}
		\rho_-=v/r,\quad\rho_0=r,\quad\rho_+=-u/r.
	\end{equation}
	We measure decay rates towards the zero set of these 3 functions.
	Let us also introduce the following set of operators acting on $C^\infty$ functions\footnote{The reader not familiar with edge vector-fields $\Ve$---as for instance in \cite{hintz_stability_2020}---may disregard all appearances of $\Ve$ on the first read,  and see their relevance in \cref{app:eq:bound}}
	\begin{equation}
		\Vb:=\{u\partial_u|_v,v\partial_v|_u,x_i\partial_{x_j}-x_j\partial_{x_i},1\},\quad \Ve:=\{u\partial_u|_v,v\partial_v|_u,\rho_+^{1/2}\rho_-^{1/2}(x_i\partial_{x_j}-x_j\partial_{x_i}),1\}.
	\end{equation}
    We use the following multi-index notation:
    \begin{notation}
        For $\mathcal{V}$ a finite set of differential operators and $X$ a function space, we write
        \begin{equation}
            \norm{\mathcal{V}^kf}_X:=\sum_{j\leq k,\Gamma_i\in\mathcal{V}}\norm{\Gamma_1...\Gamma_j f}_{X}.
        \end{equation}
    \end{notation}
	Let us also define the spaces $\A^{k;\vec{a}}$ as spaces of $C^k$ functions on $\D_{u,v}$ with the norm below finite:
		\begin{align}
			\norm{f}_{\A^{k;a_-,a_0,a_+}(\D_{u,v})}&:=\norm{w\Vb^kf}_{L^\infty(\D_{u,v})}\sim \sum_{\abs{\alpha}\leq k} \sup_{\D_{u,v}} (w\Vb^{\alpha}f),\quad w=\prod_{\bullet}\rho_\bullet^{-a_\bullet}.\label{not:eq:L_infty}
		\end{align}
	We also define function spaces on the initial data slices $\C,\Cbar$
	\begin{nalign}
		\norm{f}^2_{\A^{k;a}(\C)}&:=\norm{r^{-a}\Vb^k f}_{L^\infty(\C)};\\
		\norm{g}^2_{C^{k}(\Cbar)}&:=\norm{\{\partial_v,x_i\partial_{x_j}-x_j\partial_{x_i}\}^\alpha f}_{L^\infty(\Cbar)};\\
		\norm{g}^2_{\A^{k;a}(\Cbar)}&:=\norm{v^{-a}\Vb^k f}_{L^\infty(\Cbar)}.
	\end{nalign}
    Finally, let us write $\A^{k;\vec{a}-}=\cup_{\epsilon>0}\A^{k;\vec{a}-\epsilon}$.
    
	
	\section{The existence framework}\label{sec:local}
    In this section, we prove the existence statements relevant for ultimately proving \cref{in:thm:main}. More precisely, in \cref{subsec:3.1}, we recast the main existence result of \cite{kadar_scattering_2025}, recalled in \cref{app:thm:KK25}, as a local problem for scalar equations with a single variable. 
    In \cref{subsec:3.2}, we then also discuss the framework necessary for the extendibility across the Cauchy horizon.
    Together, these two suffice to already deduce \cref{in:thm:main} in \cref{sec:appandproof} (and the reader may therefore wish to jump there directly after having finished \cref{subsec:3.2}).  
    
    On the other hand, the results discussed in \cref{sec:3:scrip,sec:lower_bound}, pertain to \cref{rem:global,rem:sharP}, respectively, and are written to only cover the energy-critical Type II solutions of \cref{in:conj:type2}:
    In \cref{sec:3:scrip}, we discuss how to extend the solution all the way to $v=\infty$.
   In \cref{sec:lower_bound}, we discuss how to obtain the precise (ir-)regularity towards the future Cauchy horizon in certain special cases.

    \subsection{Local existence up to the Cauchy horizon \texorpdfstring{$\CH$}{CH} via \texorpdfstring{\cite{kadar_scattering_2025}}{[KK25]}}\label{subsec:3.1}
	
	Even though \cite{kadar_scattering_2025} also applies for systems, we present here a simplified statement.
	We use the convention that $\vec{b}<\vec{a}$ to mean that each component of $\vec{b}$ is smaller than $\vec{a}$, i.e. $b_-<a_-$, $b_0<a_0$, $b_+<a_+$.
	\begin{definition}\label{scat:def:admissible}
		Fix $\vec{a}\in\R^3$ and $w^{-1}=\rho_-^{a_-}\rho_0^{a_0}\rho_+^{a_+}$.
		Let $\mathcal{N}:\R_\phi\times \R^{d+1}\to\R$ be smooth in the $\phi$ variable, and smooth in $\D_{-1,\infty}\setminus \{r=0\}$.
		We say that $\mathcal{N}$ is a $k-$regular \emph{admissible nonlinearity} for $\vec{a}$ if there exists $\delta>0$, called gap, such that for any two functions $\phi_1$, $\phi_2$ with $\norm{\Ve \phi_i}_{\A^{k;\vec{a}}(\D_{1,0})}\leq C$ for $i\in\{1,2\}$, for any $j\leq k$ it holds pointwise that
        \begin{subequations}
            \begin{align}
                w\abs{\Vb^j(\mathcal{N}[\phi_2])}\rho_-^{1-\delta}\rho_0^{2-\delta}\rho_+^{1-\delta}&\lesssim_{k,C} \abs{w\Vb^j\Ve\phi_2}^2\label{scat:eq:Nquad}\\
            w\abs{\Vb^j(\mathcal{N}[\phi_1]-\mathcal{N}[\phi_1+\phi_2])}\rho_-^{1-\delta}\rho_0^{2-\delta}\rho_+^{1-\delta}&\lesssim_{k,C} \abs{w\Vb^j\Ve\phi_1}\abs{w\Vb^j\Ve\phi_2}.
            \end{align}
        \end{subequations}

		For a function $V:\R^{d+1}\to\R$, we say that it is $k-$regular \emph{admissible potential} with gap $\delta>0$ if $V\in\A^{k;\vec{a}^V}(\D_{v,0})$ for some $a_0^V\geq-2$ and $a^V_\pm\geq-1+\delta$.
	\end{definition}
	
	Let us now turn to the problems we study.
	\begin{definition}
			Fix $k\in\N$, $a_0\in\R$ and $V,\mathcal{N}$ $\infty$ regular admissible with respect to the weight $(0,a_0,0)$.
			We say that $(\phi_{\C},\phi_{\Cbar})$ are \emph{compatible} initial data for the characteristic initial value problem on $\C\cup\Cbar$ 
			\begin{equation}\label{loc:eq:characteristic}
				\Box\phi=\mathcal{N}(\phi)+V\phi+f,\quad \phi|_{\C}=\phi_{\C},\, \phi|_{\Cbar}=\phi_{\Cbar},
			\end{equation}
			with \emph{decay rate $a_0$ and regularity $k$} if $T^l\phi|_{\Cbar}\in C^{k-l}(\Cbar)$ and $(rT)^l\phi|_{\C}\in\A^{k-l;a_0}(\C)$ for $l\leq k$, where the transversal derivatives of $\phi$ satisfy transport equations along $\C,\Cbar$ and match at the intersection $\Cbar\cap\C$.
			
			We say that $(\phi|_{\C},\phi|_{\Cbar})$ are \emph{conormal compatible} initial data with \emph{decay rate $a_0$ and regularity $k$} if the inclusions $T^l\phi|_{\Cbar}\in C^{k-l}(\Cbar)+\A^{k-l;a_0+\frac{d-1}{2}-l+\delta}(\Cbar)$ and $(rT)^l\phi|_{\C}\in\A^{k-l;a_0}(\C)$ hold for all $l\leq \min({a_0+\frac{d-1}{2}},k)$ and some $\delta>0$.
	\end{definition}

    Provided a solution exists in the region $\mathcal{X}$---as in \cref{in:conj:type2}---the above definition is related to the uniform $a_0$ decay rate  of $\phi$ in the region $\D_{\mathrm{in}}=\{v/r\in[-1/10,0],t>-1\}$.
    More precisely, let us define conormal and smooth function spaces via the norms
    \begin{nalign}
        \norm{f}_{\A^{k;a_0,a_-}(\D_{\mathrm{in}})}&:=\sum_{\abs{\alpha}\leq k}\sup_{\D_{\mathrm{in}}} r^{-a_0}\rho_-^{-a_-} \Vb^\alpha f,\\
        \norm{f}_{\Aext^{k;a_0}(\D_{\mathrm{in}})}&:=\sum_{\abs{\alpha}\leq k}\sup_{\D_{\mathrm{in}}} r^{-a_0}\big((\Vb\cup\{r\partial_v\})^\alpha f\big).
    \end{nalign}
    The following regularity in the interior of the lightcone suffices to construct compatible data:
    \begin{lemma}\label{scat:reverse_peeling}
        Fix some nonlinearity and potential $\mathcal{N},V$, together with $a_0,a_-\in\R$ satisfying $a_->a_0+\frac{d-1}{2}$ such that for every
        $\phi\in\A^{\infty;a_0,a_-}(\D_{\mathrm{in}})+\Aext^{\infty;a_0}(\D_{\mathrm{in}})$ it holds that $\mathcal{N}[\phi]\in\A^{\infty;a_0-2,a_--1+\epsilon}(\D_{\mathrm{in}})+\Aext^{\infty;a_0-2}(\D_{\mathrm{in}})$, and $V\in\Aext^{\infty;-2}(\D_{\mathrm{in}})+\A^{\infty;-2,-1+\epsilon}(\D_{\mathrm{in}})$ for some $\epsilon>0$.
        Let $\phi\in\A^{\infty;a_0,a_-}(\D_{\mathrm{in}})+\Aext^{\infty;a_0}(\D_{\mathrm{in}})$ be a solution of $\Box\phi=\mathcal{N}[\phi]+V\phi$ in $\D_{\mathrm{in}}$.
        
        Then, there exists data on the outgoing cone $\phi_{\Cbar}\in C^{\infty}(\Cbar)+\A^{\infty;a_-}(\Cbar)$ such that $(\phi|_{\C},\phi_{\Cbar})$ are $k$ regular conormal compatible initial data with decay $a_0$ for all $k\geq0$.
        Provided that $\phi\in\Aext^{\infty;a_0}(\D_{\mathrm{in}})$, there exist $\phi_{\Cbar}\in C^\infty(\Cbar)$ such that $(\phi|_{\C},\phi_{\Cbar})$ are compatible initial data with decay $a_0$.
    \end{lemma}
    \begin{proof}
        By assumption, it holds that $(r\partial_t)^j\phi|_{\Cbar}$ has the required regularity for all $j$ in the smooth case and for $j\leq a_-$ in the conormal case.
        The result follows by extendibility of Sobolev functions \cite{evans_partial_2010}[Section 5.4].

    \end{proof}

    We can use the compatibility definition to find an approximate solution to \cref{loc:eq:characteristic} with an error term vanishing to high order as $\rho_-\to0$ \textit{and} as $\rho_0\to0$.
	\begin{lemma}[Peeling]\label{scat:peeling}
		Fix $v>0,u=-1$, as well as $l\in\N$, $\vec{a}\in\R^3$ and $k>l+1$ with $a_-=0$.
		Fix $f=0$ and $\mathcal{N},V$ admissible nonlinear and linear functions in $\D_{-1,1}$, and $\phi_{\C},\phi_{\Cbar}$ smooth compatible initial data with decay $a_0$ and regularity $k$.
		Then there exists $\phi_1\in\A^{s-l;0,a_0,\infty}(\D_{-1,1})$ with $(rT)^j\phi_1\in\A^{k-l-j;0,a_0,\infty}(\D_{-1,1})$ for any $k\leq s$ such that
		\begin{equation}\label{eq:peeling}
			\Box\phi_1-\mathcal{N}[\phi_1]-V\phi_1\in\A^{k-l-2;l,a_0-2,\infty}(\D_{-1,1}), \quad \phi_1|_{\C}=\phi_{\C},\, \phi_1|_{\Cbar}-\phi_{\Cbar}\in\A^{k-l;l+1}(\Cbar).
		\end{equation} 
		Moreover, $\mathcal{N}_{\phi_1}[\phi]:=\mathcal{N}[\phi+\phi_1]-\mathcal{N}[\phi_1]$ is admissible with respect to $\vec{a}$ and thus also with respect to $(l,a_0,a_+)$.

        The same conclusion holds for conormal compatible initial data,  as long as $l\leq a_0+\frac{d-1}{2}$ and the estimate on $\phi_1|_{\Cbar}-\phi_{\Cbar}$ is weakened to $\phi_1|_{\Cbar}-\phi_{\Cbar}\in \A^{k-l;\min(l+1,a_0+\frac{d-1}{2}+\delta)}(\Cbar)$. 
	\end{lemma}
	\begin{proof}
		Let's pick a cut-off function $\chi(s)$ localising to $s<-1$.
		We simply set
		\begin{equation}
			\phi_1(v,u,x/\abs{x})=\chi(v/u)\sum_{k=0}^{l} \frac{v^k}{k!} (\partial_v^k\phi)|_{\C}(u,x/\abs{x})\in\A^{s-l;0,a_0,\infty}(\D_{-1,1}).
		\end{equation}
		Both estimates of \cref{eq:peeling} then follow simply by Taylor expanding 
        \cref{loc:eq:characteristic} in $v$ and using the smoothness of $\mathcal{N}$ and $V$ across $\C$. (Recall that the quantities $(\partial_v^j\phi)|_{\C}$ are themselves determined by integrating the equation \cref{loc:eq:characteristic} along $\C$.)
        
		The admissibility of $\mathcal{N}_{\phi_1}$ follows by definition.
	\end{proof}

	At this stage, we may cite the following existence result (which we discuss in more detail in the appendix):
	\begin{prop}[\cite{kadar_scattering_2025}]\label{scat:cor:local}
		Fix $\vec{a}\in\R^3$ with $a_->0$, $\epsilon>0$ and $k\geq \frac{d+2}{2}+5$.
		Fix $\mathcal{N},V$ admissible with gap $\delta>0$.
		Furthermore, fix $f\in\A^{k;\vec{a}^f}(\D_{-1,1})$ and $\phi_{\C}=0$, $\phi_{\Cbar}\in\A^{2k;a_-}(\Cbar)$.
		Assume that $(a_-,a_0,a_+)$ satisfies $a_+<0$ and
		\begin{equation}\label{scat:eq:admiss}
			a_->a_0+\frac{d-1}{2}>a_+,\quad \vec{a}\leq \vec{a}^f+(1,2,1).
		\end{equation} 
		Then, there exists $v_1>0$, depending inversely on the size of $f,\phi_{\Cbar},V,\mathcal{N}$, such that \cref{loc:eq:characteristic} has a solution in $\D_{-1,v_1}$. 
		In $\D'=\D_{-1,v_1}\setminus\D_{-1,v_1/2}$, we have the following estimate: 
		\begin{equation}
			\phi\in\A^{k-\frac{d}{2};a_+-\epsilon}(\D'),\quad \forall\epsilon>0.
		\end{equation}
		Moreover, if $a_0+\frac{d-1}{2}>0$, then we also have $\phi\in\A^{k';a_0+\frac{d-1}{2}}(\D')+C^{k'}(\D')$ for some $k'\sim_{d,a_0} k$.
	\end{prop}	
	\begin{proof}
		We begin by rewriting the characteristic problem with coordinates $\underline{u}=-v^{-1}$, $\underline{v}=-u^{-1}$, $\underline{r}=\underline{v}-\underline{u}$.
		Then, we write $\Box$ as
		\begin{equation}
			\Box\phi=\left(-\partial_u\partial_v+\frac{d-1}{2r}(\partial_v-\partial_u)+\frac{\mathring{\slashed{\Delta}}}{r^2}\right)\phi=\frac{\underline{u}^2\underline{v}^2}{r^{\frac{d-1}{2}}}\underbrace{\Big(-\partial_{\underline{u}}\partial_{\underline{v}} +\frac{\mathring{\slashed{\Delta}}-(d-1)(d-3)/4}{\underline{r}^2}\Big)}_{:=\mathcal{P}_d}(r^{\frac{d-1}{2}}\phi),
		\end{equation}
        where $\mathring{\slashed{\Delta}}$ is the Laplacian on the unit sphere.
		Therefore, we conclude for $\psi=r^{\frac{d-1}{2}}\phi$ and $\phi^{\mathrm{g}}=\underline{r}^{-\frac{d-1}{2}}\psi$ that
		\begin{equation}
			\mathcal{P}_d\psi=\frac{r^{\frac{d-1}{2}}}{\underline{u}^2\underline{v}^2}\mathcal{N}[\phi]+\frac{1}{\underline{u}^2\underline{v}^2}V\psi\implies \underline{\Box}\phi^{\mathrm{g}}=\frac{\underline{r}^{-\frac{d-1}{2}}r^{\frac{d-1}{2}}}{\underline{u}^2\underline{v}^2}\mathcal{N}[\phi]+\frac{1}{\underline{u}^2\underline{v}^2}V\phi^{\mathrm{g}},
		\end{equation}
        where $\underline{\Box}=-\partial_{\underline{u}}\partial_{\underline{v}}+\frac{d-1}{2\underline{r}}(\partial_{\underline{v}}-\partial_{\underline{u}})+\underline{r}^{-2}\mathring{\slashed{\Delta}}$.
		We can apply the result of \cite{kadar_scattering_2025} as recalled in \cref{app:lemma:equivalent,app:thm:KK25} to deduce the proposition.
	\end{proof}

    Now, \cref{scat:cor:local} requires the relation $a_->a_0+\frac{d-1}{2}$ to be satisfied for the solution constructed.
    As this will in general not hold for the solutions in \cref{in:thm:main}, we need to remove the assumption on $a_-$ by peeling some leading order terms via \cref{scat:peeling} to deduce:
	\begin{cor}\label{scat:cor:main}
        a)
		Fix $k>d+5+a_0$.
		Fix $\vec{a}\in\R^3$ with $a_+<0$ and corresponding admissible $\mathcal{N},V$. 
		Let $\phi_{\C},\phi_{\Cbar}$ be compatible or conormal compatible initial data with decay $a_0$ and regularity $k$.
		Then there exist $v_1>0$ and a solution of \cref{loc:eq:characteristic} in $\D_{-1,v_1}$.

        b)
        Moreover, we have the following estimate in $\D'=\D_{-1,v_1}\setminus\D_{-1,v_1/2}$:
		\begin{equation}
			\phi\in\A^{k-\frac{d}{2};a_+}(\D').
		\end{equation}
		Furthermore, if $a_0+\frac{d-1}{2}>0$, then for some $k'\sim_{d,a_0} k$, we also have
        \begin{equation}\label{scat:eq:improved_reg}
            \phi\in\A^{k';a_0+\frac{d-1}{2}-}(\D')+C^{k'}(\D').
        \end{equation}
	\end{cor}
	\begin{proof}
		This follows from \cref{scat:peeling} and  \cref{scat:cor:local}.
        More precisely, we first peel off $l=\lfloor a_0+\frac{d-1}{2}\rfloor$ terms using \cref{scat:peeling}, and then apply \cref{scat:cor:local} to the difference $\phi-\phi_1$.
        \end{proof}

    \subsection{Extendibility beyond the Cauchy horizon}\label{subsec:3.2}
    So far, we have only shown that the solution exists up to the null hypersurface $\CH$.
    However, we can show that it can be (non uniquely) extended across $\CH$ away from $r=0$ such that it is still a weak solution of the equation.
    As the solutions of \cref{in:conj:type1,in:conj:type2} are no longer smooth (respectively $C^{c_\nu}$ regular) up to and including $(r,t)=(0,0)$ this shows that $\CH$ is a Cauchy horizon.\footnote{It is important to note that while $\CH$ is the boundary of the maximal globally hyperbolic development in a given regularity class, it may or may not be the case that there is a selection criterion to uniquely continue singularity forming solutions of \cref{in:conj:type1,in:conj:type2} across $\CH$ in a rougher topology.}

    \begin{lemma}
        Let $\phi,k$ be as in \cref{scat:cor:main} with $a_0+\frac{d-1}{2}>0$ and $k$ sufficiently large so that \cref{scat:eq:improved_reg} holds for $k'\geq d+10$.\footnote{We are not working with the sharp regularity readily available from the literature.}
        Then, $\phi$ can be extended as a weak solution to $P[\phi]=0$ across $\CH$ away from $r=0$.
    \end{lemma}
    \begin{proof}
        Let $v_1$ be as in \cref{scat:cor:main}.
        Observe that $\phi|_{\CH}$ is $C^{k'}(\CH)$ as a function restricted onto $\CH$, as the $\A$ part is absent for this restriction. Let us write $P[\phi]=\Box\phi-\mathcal{N}[\phi]-V\phi-f$. 
        From local well-posedness of the characteristic initial value problem, there exists a solution $\bar\phi\in H^{k'/2}$ to
        \begin{equation}
            P[\bar\phi]=0,\quad \bar\phi|_{\CH}=\phi|_{\CH},\, \bar\phi|_{v=v_1/2,u>0}(u,x/r)=\phi|_{\CH\cup \{v=v_1/2\}}(x/r)
        \end{equation}
        in the region $\{u\in(0,\epsilon),v\in(v_1/2,v_1/2+\epsilon)\}$ for $\epsilon\ll 1$ \cite{rendall_reduction_1990}.
        Let us define in $\{u\in(-\epsilon,\epsilon),v\in(v_1/2,v_1/2+\epsilon)\}$
        \begin{equation}
            \tilde\phi=\begin{cases}
                \phi& u<0\\
                \bar\phi & u>0.
            \end{cases}
        \end{equation}
        It is easy to check that $\tilde\phi$ is a weak solution to $P[\phi]=0$, since $P[\psi]\in L^1$ for all $\psi$ such that
        \begin{equation}
            \psi=\begin{cases}
                \psi_-\in\A^{k';a_0+\frac{d-1}{2}-}(\D_{-1,v_1})+C^{k'}(\D_{-1,v_1})& u<0\\
                \psi_+\in H^{k'/2}& u>0,
            \end{cases}\qquad
            \psi_+|_{\CH}=\psi_-|_{\CH}.\qedhere
        \end{equation}
    \end{proof}

    \subsection{Semi-global existence: Extending to future null infinity. }\label{sec:3:scrip}

    The previous two subsections are applicable to the solutions of \cref{in:conj:type1,in:conj:type2,in:conj:type2_sup}. In the current subsection, we provide the statements necessary to prove the extendibility of only the solutions of \cref{in:conj:type2} (which have $\phi_0=0$ as a background) to $\scrip$, i.e.~to $\{v=\infty\}$.\footnote{The case of a non-trivial background as in \cref{in:conj:type1,in:conj:type2_sup} requiring different techniques, see \cite{christodoulou_examples_1994,rodnianski_naked_2023} for details.}
    To use standard semi-global existence results, we require that the Cauchy horizon is sufficiently regular.
    We emphasise that this requirement is only necessary to appeal to \cite{kadar_scattering_2025}, however the methods should also yield global existence in all cases where finite time existence can be  proved using in \cref{scat:cor:main}.

    We begin by noting that a sufficiently regular interior solution---as assumed in \cref{in:conj:type2}---can be extended along $\Cbar_u$ with small data simply by taking $u$ sufficiently close to 0:
    \begin{lemma}
        Let $\mathcal{N},V,a_0,a_-,\phi$ be as in \cref{scat:reverse_peeling}.
        Then, for any $k\in\N$, there exists $\phi_{\Cbar_u}$ with $\supp\phi_{\Cbar_u}\subset \Cbar_u\cap\{v<1\}$ such that $(\phi|_{\C},\phi_{\Cbar_u})$ is conormal admissible on $\C\cup\Cbar_u$ with $\phi_{\Cbar_u}=\phi_{\Cbar_u,\mathrm{s}}+\phi_{\Cbar_u,\mathrm{c}}$ and
        \begin{equation}\label{scat:eq:small_extension}
            \norm{\phi_{\Cbar_u,\mathrm{c}}}_{\A^{k;a_-}(\Cbar_u)},\norm{\{1,r\partial_v,x_i\partial_{x_j}-x_j\partial_{x_i}\}^k\phi_{\Cbar_u,\mathrm{s}}}_{L^\infty(\Cbar_u)}\lesssim_{k} \abs{u}^{a_0},
        \end{equation}
        with implicit constant independent of $u$.
    \end{lemma}
    In the rest of this section, we study solutions to 
    \begin{equation}\label{loc:eq:characteristic_u}
				\Box\phi=\mathcal{N}(\phi)+V\phi+f,\quad \phi|_{\C}=\phi|_{\C},\, \phi|_{\Cbar_u}=\phi_{\Cbar_u},
    \end{equation}
    with the data assumed to satisfy the bound  \cref{scat:eq:small_extension}.
    
    The only important boundary hypersurface to consider is $\scrip=\{v=\infty\}$, and we define the corresponding function spaces
    \begin{equation}
        \norm{f}_{\A^{k;a}(\D^{\scrip})}=\norm{r^{a}\Vbs^k f}_{L^\infty(\D^{\scrip})},\qquad \Vbs:=\{x_i\partial_{x_j}-x_j\partial_{x_i},v\partial_v,\partial_u\}.
    \end{equation}
    Note that these function spaces are oblivious to the actual conormal regularity proved in \cref{scat:cor:main}.
    The relevant notion of admissibility in $\D^{\scrip}:=\D_{-1,\infty}$ is
    \begin{definition}
        Fix $a\in\R_{<\frac{d-1}{2}}$.
        We say that a nonlinearity $\mathcal{N}$ is admissible near null infinity with gap $\delta>0$ and regularity $k$ if for all $\phi_1, \phi_2$ with  $\norm{\phi_1,\phi_2}_{\A^{k;a}(\D^{\scrip})}\leq C$ and $j\leq k$ it holds that 
        \begin{subequations}
            \begin{align}
                r^{a+1+\delta}\abs{\Vbs^j\mathcal{N}[\phi_1]}&\lesssim_{k,C} \abs{r^a\Ve\Vbs^j\phi_1}^2,\\
            r^{a+1+\delta}\abs{\Vbs^j\mathcal{N}[\phi_1+\phi_2]}&\lesssim_{k,C} \abs{r^a\Ve\Vbs^j\phi_1}\abs{r^a\Vbs^j\Ve\phi_2}.
            \end{align}
        \end{subequations}
        We call a potential admissible near null infinity if $V\in \A^{k;1+\delta}(\D^{\scrip})$.
    \end{definition}

    In sufficiently high regularity, we may conclude existence up to $\scrip$.
    \begin{cor}\label{scat:cor:null_inf}
        Let $k'\geq 5+\frac{d-1}{2}$ and let $\mathcal{N},V$ be admissible for $a<\frac{d-1}{2}$.
        Fix $v_1>0$.
        There exists $\epsilon_0$ sufficiently small depending on $\mathcal{N},V,v_1$ such that the following holds:
        Let $\phi$ be a solution of \cref{loc:eq:characteristic} in $\D_{v_1}\setminus\D_{v_1/2}$ with $\psi_{\Cbar}|_{v>v_1/2}=0$
        \begin{equation}
            \norm{\phi}_{C^{k'}(\D_{v_1}\setminus\D_{v_1/2})}\leq \epsilon\leq\epsilon_0.
        \end{equation}
        Then $\phi$ admits a solution to \cref{loc:eq:characteristic} in $\D^{\scrip}$
    \end{cor}
    \begin{proof}
        This follows from \cref{app:thm:KK25}.
    \end{proof}

    Combining with the local existence result, we obtain existence all the way to $\scrip$:
    \begin{cor}\label{scat:cor:scrip}
        Fix $a_0>5$ and $k$ sufficiently large so that $k'$ from \cref{scat:cor:main} satisfies $k'>5+\frac{d-1}{2}$.
        Let $\vec{a}=(a_0,a_0,a_+)$ for some $a_+<0$, and fix corresponding admissible $\mathcal{N},V$ and compatible initial data $\phi_{\C},\phi_{\Cbar_u}$ satisfying \cref{scat:eq:small_extension}.

        Then, if $u<0$ is sufficiently close to 0, the solution $\phi$ of \cref{loc:eq:characteristic_u}  exists in $\D_{u,\infty}$.
    \end{cor}
    \begin{proof}
        We may pick $a_0'<a_0$, such that $\mathcal{N}$ is still admissible with respect to $(a_0',a_0',a_+)$ and moreover the data along $\C$ satisfies $\norm{\phi}_{\A^{k;a_0'}(\C\cap\{u>u_0\})}\lesssim \abs{u_0}^{a_0-a_0'}$.
        Using \cref{scat:eq:small_extension}, we may obtain a solution from \cref{scat:cor:main} in $\D_{v_1,0}$ with $v_1$ independent of $u$ with the norms in \cref{scat:eq:improved_reg} bounded by $\abs{u}^{a_0-a_0'}$.
        If $|u|$ is sufficiently small, then we may apply  \cref{scat:cor:null_inf} to conclude the result.
    \end{proof}

    \subsection{Lower bound on regularity of \texorpdfstring{$\CH$}{the Cauchy horizon}}\label{sec:lower_bound}
    Following the discussion in \cref{rem:sharP}, we finally showcase how to study the sharp regularity of the Cauchy horizon in the special case of $d=3$ spherical symmetry\footnote{These restrictions are merely to ease notation and to get a simple expression for the approximate solution $\phi_{\mathrm{app}}$ below. See \cite[Section 10]{kadar_scattering_2025} for the general case.} and no potential with $\rho_0^{-2}$ behaviour.\footnote{Such borderline potentials appear for Type I and energy-supercritical Type II singularities. However, in the latter case the potential has the exact form $cr^{-2}$ for some $c\in\R$ and thus the approach computations of \cite{kadar_scattering_2025} suffice.}
    That is, in this section we assume $\mathcal{N}$ to be an admissible nonlinearity with gap $\delta>0$ and $V\in \A^{\infty;0,2+\delta,0}(D_{-1,1})$.
    We also assume that $a_0+1>0$.
    These assumptions in particular cover \cref{nw:d=3}.
    Next, we follow the more general \cite[Section 7-8]{kadar_scattering_2025}, applicable outside of symmetry in all dimensions.

    Let us introduce the function spaces measuring smoothness towards $\CH$ but only conormal regularity towards $\{\rho_0=0\}$ in $\D^+=\D_{-1,1}\cap\{t>0\}$
    \begin{equation}
        \norm{f}_{\Aext^{k;a_0}(\D^+)}:=\norm{\rho_0^{-a_0}\{r\pu,r\pv\}f}_{L^\infty(\D^+)}.
    \end{equation}
    
    Fix $\phi_{\C}$ in \cref{loc:eq:characteristic} satisfying an additional lower bound:
    \begin{equation}\label{scat:eq:lower_bound}
        \phi_{\C}=r^{a_0}+\A^{\infty;a_0+\delta}(\C).
    \end{equation}
    We may subtract the leading order behaviour by noting that $\phi_{\mathrm{app}}=r^{-1}\abs{u}^{a_0+1}$ is a solution of $\Box\phi_{\mathrm{app}}=0$.
    Thus, we can reduce \cref{loc:eq:characteristic} for $\bar\phi=\phi-\phi_{\mathrm{app}}$ under the assumption of \cref{scat:eq:lower_bound} to
    \begin{equation}
        \Box\bar\phi=\mathcal{\bar\phi+\phi_{\mathrm{app}}}+V\bar\phi+V\phi_{\mathrm{app}}+f,\qquad\bar\phi|_{\C}=\phi_{\Cbar}-r^{a_0},\,\bar\phi|_{\Cbar}=\phi_{\Cbar}-r^{-1}.
    \end{equation}
    Let us observe that the additional inhomogeneous terms coming from $\phi_{\mathrm{app}}$ satisfy
    \begin{equation}
        V\phi_{\mathrm{app}},\mathcal{N}[\phi_{\mathrm{app}}]\in \A^{\infty;\delta-1,\delta+a_0-2,\delta+a_0}(\D_{-1,1})+\Aext^{\infty;\delta+a_0-2}(\D^+).
    \end{equation}
    The result of \cite[Theorem 8.1]{kadar_scattering_2025} applies also in this setting to conclude that $\bar\phi$ satisfies the following improvement over \cref{scat:eq:improved_reg} in the region $\D'=\D_{-1,v_1}\setminus \D_{-1,v_1/2}$
    \begin{equation}\label{scat:eq:sharp}
        \bar\phi\in \A^{\infty;a_0+\delta+1-}(\D')+C^\infty(\D').
    \end{equation}
    Therefore, for $a_0\notin\N$ the derived upper bound of regularity in \cref{scat:eq:improved_reg} is sharp.
    We summarise the above as:
    \begin{lemma}\label{scat:lem:sharp}
        Let $d=3$, $a_0+1>0$, $k=\infty$ and $\vec{a},\phi_C,\phi_{\Cbar},\mathcal{N},V,\phi$ be as in \cref{scat:cor:main}.
        Assume moreover that $\phi_{\C}$ satisfies the lower bound in \cref{scat:eq:lower_bound} and $\phi_{\C},\phi_{\Cbar}$ are spherically symmetric.
        Then the solution $\bar\phi=\phi-r^{-1}\abs{u}^{a_0+1}$ satisfies the sharp regularity \cref{scat:eq:sharp} towards $\CH$.
        In particular, if $a_0+1\notin\N$, then $\phi\notin C^{a_0+1+}$.
    \end{lemma}

	\section{Applications and proof of \texorpdfstring{\cref{in:thm:main}}{the main theorem}}\label{sec:appandproof}
    In this section, we use \cref{scat:cor:main} to prove \cref{in:thm:main}.
    To do so, we need to find an appropriate vector $\vec{a}\in\R^3$ and show that the corresponding nonlinearity is admissible. We will do this first for the power nonlinear wave equation, then for the wave maps equation, where for the latter, we restrict to corotational perturbations. 
    
	\subsection{The power nonlinear wave equation\texorpdfstring{: $\Box\phi+\phi^p=0$}{}}
    We first consider Type II blow-up for \cref{in:eq:nonlin} and start with case \cref{nw:d=3} for which $d=3$, $p=5$.
    First, the nonlinearity $\phi^5$ is admissible for $\vec{a}=(0,a_0,0)$ so long as $a_0>-1/2$; indeed, admissiblity requires that $5a_0+2>a_0$.
    In order to be able to apply \cref{scat:cor:local}, we need to peel off sufficiently many terms via \cref{scat:peeling} so that the peeled off solution has $a_->a_0+\frac{d-1}{2}=a_0+1>1/2$. 
    This is precisely the point of \cref{scat:cor:main}:
    Thus, given any $\nu>1$, we may apply \cref{scat:cor:main} with any $a_0>-\frac12$ satisfying $\frac{\nu}{2}-1>a_0$.
    This finishes case \cref{nw:d=3} since $\frac{\nu-2}{2}>a_0$.

    We move on to \cref{nw:d=4}, for which $d=4$, $p=3$. 
    Admissibility leads to the requirement that $a_0>-1$.
    Thus, given $\nu>1$, we may apply \cref{scat:cor:main} with any $a_0>-1$ satisfying $\frac{\nu}{2}>a_0+\frac32$.
    This finishes case \cref{nw:d=4} since $\nu-2>a_0$.

    Finally, we move on to Type I blow-up, i.e.~to \cref{type1:wave}, for which $p=3$, $d=7$. Here, we are perturbing around the solution $\phi_0\in\A^{0,-2,0}(\D_{-1,v_1})$ of \cref{in:thm:self_similar_existence} b), that is, $\bar\phi=\phi-\phi_0$ satisfies: 
     \begin{equation}
        \Box{\bar\phi}+3\phi_0^2\bar\phi+3\phi_0\bar\phi^2+\bar\phi^3=0.
    \end{equation}
The $\phi_0^2\in\A^{0,-2,0}(\D_{-1,1})$-term is an  admissible potential, and the nonlinearities $\phi_0\cdot \bar\phi^2,\bar\phi^3$ are admissible for $(0,a_0,0)$ with $a_0>-1$.
Since the solution is smooth across $\C$ along $\Cbar$, the result now follows.
The computation for \cref{nw:supercritical} is identical, with $a_0>-\frac{2}{p-1}$ and potential $\phi_0^{p-1}\in\A^{0,-2,0}(\D_{-1,1})$.

    \begin{remark}[Logarithmic admissibility 1]\label{applications:rem}
    We have not yet addressed case \cref{nw:d=4_log}; this is because the functional framework of \cite{kadar_scattering_2025} works with polynomially weighted function spaces and makes use of the admissibility \cref{scat:def:admissible} via Sobolev embedding, the one dimensional version of which is
        \begin{equation}
            \norm{f}_{L^\infty((0,1])}\lesssim \int_{(0,1]} \left(\{x\partial_x,1\}\partial f\right)^2 \frac{\dd x}{x}\qquad\forall f\in C^\infty((0,1]).
        \end{equation}
        In fact, one may replace $\rho_-^{-\delta}\rho_0^{-\delta}\rho_+^{-\delta}$ with a sufficiently small multiplicative factor $\epsilon\ll1$ in \cref{scat:eq:Nquad}.
        Note that for $a_0=0$ it holds that
        \begin{equation}
            \int \frac{\dd u}{u} e^{-2\sqrt{|\log(r)|}}<\infty,
        \end{equation}
        and so  \cite{kadar_scattering_2025}[Theorem 6.2] implies the existence of a solution in $\D_{-1,v_1}$ for \cref{nw:d=4_log}.

        \end{remark}
        
	\subsection{The wave maps equation in corotational symmetry}    
    Recall that in the $K=1$-corotational setting \cref{in:eq:wave_map} yields
	\begin{equation}
		(-\partial_t^2+\partial_r^2+\frac{d-1}{r}\partial_r)\tilde{\phi}-\frac{d-1}{2}\frac{\sin(2\tilde\phi)}{\abs{x}^2}=0.
	\end{equation}
	We separate the linear and the nonlinear parts as $\mathcal{N}[\phi]=\frac{d-1}{2}r^{-2}(\sin(2\phi)-2\phi)$ and $V\phi=\frac{d-1}{2}r^{-2}2\phi$. 
    Note that $V$ is always an admissible potential.
	
	The exterior stability of Type II singularity formation follows from noting that:
	\begin{equation}
		\phi\in\A^{0,a_0,0}(\D_{-1,1})\implies\mathcal{N}[\phi]\in\A^{0,2a_0-2,0}(\D_{-1,1}).
	\end{equation}
	Therefore the nonlinearity is clearly admissible for $a_0>0$.
	This covers \cref{wm:d=2_poly}.
   
    To also cover \cref{wm:d=2}, for which there is only finite regularity $\nu-1/2$, we have the additional condition that $\nu>a_0+1>1$.
	
	For Type I singularity formation, there is a background solution $\phi_0$, and thus we must write 
	\begin{equation}
		\frac{\sin(2\phi_0+2\phi)}{2r^2}=\frac{\cos(2\phi_0)}{r^2}\phi+\mathcal{N}_{\phi_0}[\phi].
	\end{equation}
	It is easy to check that $\mathcal{N}_{\phi_0}$ again is admissible for $(0,\nu,0)$ with $\nu>0$. The linear term is also admissible: $\frac{\cos(2\phi_0)}{r^2}\in\A^{0,-2,0}(\D_{-1,1})$.
	This covers \cref{type1}.
    A similar computation applies to \cref{wm:d>7}.

    \begin{remark}[Logarithmic admissibility 2]\label{applications:rem2}
        By the same reasoning, as in \cref{applications:rem}, we can also treat \cref{wm:d=2_log}. On the other hand,   in order to study \cref{wm:d=2_log_k} one needs to consider logarithmically weighted spaces. Since this is not included in \cite{kadar_scattering_2025}, we claim no results in this direction.
    \end{remark}
    \begin{remark}\label{remark:outsideofsymmetry}
        We do not make statements on the stability outside of corotational symmetry for \cref{in:eq:wave_map}, but the interested reader may wish to study \cref{subsec:extendingadmissiblity} for ways of proving such statements.
    \end{remark}

	\pagebreak
	\appendix
    \section{The scattering result of \texorpdfstring{\cite{kadar_scattering_2025}}{[KK25]} and an extension}\label{sec:appendix}
    In this section, we recall the main existence and regularity result of \cite{kadar_scattering_2025} and the corresponding notation in a simplified setting. 
    Furthermore, we will also give a sketch of a direct proof of \cref{scat:cor:local} in \cref{sec:app:sketch}, as well as an extension of our framework relevant for \cref{in:eq:wave_map} outside symmetry in \cref{subsec:extendingadmissiblity}.
    
    Let us use $u,v$ as in \cref{sec:not} and define
    \begin{equation}
        \brho_-=v/r,\quad \brho_0=-r/uv,\quad \brho_+=-u/r.
    \end{equation}
    We consider equations 
    \begin{equation}\label{app:eq:semi}
        \Box\phi=f+\mathcal{N}(\phi,x,t)+V\phi,\quad \phi(u,v,x/r)|_{v=1}=\phi_1(r,x/r),\quad \lim_{u\to-\infty} r^{\frac{d-1}{2}}\phi(u,v,x/r)=\psi_2(v,x/r),
    \end{equation}
    in the region $\D_{\mathrm{g}}=\{u\in(-\infty,-1),v\in(1,\infty)\}$.
    Let us also define the function space $\A^{s;\vec{a}}(\D_{\mathrm{g}})$ as in~\cref{not:eq:L_infty}.

    \begin{definition}[Admissibility]\label{app:def:admissible}
    Fix $\vec{a}\in\R^3$ and $w^{-1}=\brho_-^{a_-}\brho_0^{a_0}\brho_+^{a_+}$.
        For $\mathcal{N},V$ in \cref{app:eq:semi} we define admissibility at regularity $k\in\N$ as follows:
        \begin{itemize}
            \item We call a potential $V$ admissible if $V\in\A^{k;1+\epsilon,2,1+\epsilon}(\D_{\mathrm{g}})$ for some $\epsilon>0$.
            \item We call a nonlinearity $\mathcal{N}$ admissible for decay rates $(a_-,a_0,a_+)$ if for $\phi_1$, $\phi_2$ satisfying $\norm{\phi_i}_{\mathcal{A}^{k;a_-,a_0,a_+}(\D_{\mathrm{g}})}\leq C$  for $i=1,2$, it holds that for all $j\leq k$
        \begin{subequations}
            \begin{align}
                w\abs{\Vb^j(\mathcal{N}[\phi_2])}\brho_-^{-\delta-1}\brho_0^{-\delta-2}\brho_+^{-\delta-1}&\lesssim_{k,C} \abs{w\Vb^j\Ve \phi_2}^2\label{app:eq:Nquadratic}\\
            w\abs{\Vb^j(\mathcal{N}[\phi_1]-\mathcal{N}[\phi_1+\phi_2])}\brho_-^{-\delta-1}\brho_0^{-\delta-2}\brho_+^{-\delta-1}&\lesssim_{k,C} \abs{w\Vb^j\Ve\phi_1}\abs{w\Vb^j\Ve\phi_2}.
            \end{align}
        \end{subequations}
        \end{itemize}
    \end{definition}

    It is a computation to see that the local (\cref{scat:def:admissible}) and the global admissibility  definition (\cref{app:def:admissible}) are equivalent:
    Let  and let $(\underline{u},\underline{v})=\Phi(u,v)=(-v^{-1},-u^{-1})$, $\Phi:\D_{-1,1}\to\D_{\mathrm{g}}$ be the change of coordinates used in  \cref{scat:cor:local}.
    \begin{lemma}\label{app:lemma:equivalent}
        Let $\vec{a}\in\R^3$.
        For $\phi\in A^{\vec{a}}(\D_{-1,1})$, let us define $\phi^{\mathrm{g}}=\Phi_*\phi$.
        A nonlinearity $\mathcal{N}[\phi]$ in $\D_{-1,1}$ is admissible with weight $\vec{a}$ as in \cref{scat:def:admissible} if and only if 
        \begin{equation}
            \mathcal{N}^{\mathrm{g}}[\phi^{\mathrm{g}}]=\frac{\underline{r}^{-\frac{d-1}{2}}\Phi_*(r^{\frac{d-1}{2}})}{\underline{u}^2\underline{v}^2}\Phi_*\mathcal{N}\left[r^{-\frac{d-1}{2}}\Phi^*( \underline{r}^{\frac{d-1}{2}}\phi^{\mathrm{g}}) \right]
        \end{equation}
        is admissible in $\D^{\mathrm{g}}$ with respect to $\vec{a}^{\mathrm{g}}=(a_-+\frac{d-1}{2},a_0+d-1,a_++\frac{d-1}{2})$ as in \cref{app:def:admissible}.
    \end{lemma}
    \begin{proof}
        Observe that $\Phi_*\rho_\bullet=\brho_\bullet$ for $\bullet\in\{\pm,0\}$ as well as $r=\rho_0$, $\underline{r}^{-1}=\brho_-\brho_0\brho_+$.
        We define the weights $(w^{\mathrm{g}})^{-1}=\brho_-^{a^{\mathrm{g}}_-}\brho_0^{a^{\mathrm{g}}_0}\brho_+^{a^{\mathrm{g}}_+}$ and $w^{-1}=\rho_-^{a_-}\rho_0^{a_0}\rho_+^{a_+}$.
        We only show that \cref{scat:eq:Nquad} implies \cref{app:eq:Nquadratic}, the rest following similarly.
        We compute
        \begin{multline}
            \abs{\mathcal{N}^{\mathrm{g}}[\phi^{\mathrm{g}}]}\lesssim \brho_-^{2+\frac{d-1}{2}}\brho_+^{2+\frac{d-1}{2}}\brho_0^{4+d-1}\abs{\Phi_*\mathcal{N}\left[\underline{r}^{\frac{d-1}{2}}\Phi^*(r^{-\frac{d-1}{2}} \phi^{\mathrm{g}}) \right]} \\
            \lesssim \brho_-^{2+\frac{d-1}{2}}\brho_+^{2+\frac{d-1}{2}}\brho_0^{4+d-1} \Phi^*(w^{-1} \rho_-^{-1+\delta}\rho_0^{-2+\delta}\rho_+^{-1+\delta})\abs{\Ve w^{\mathrm{g}} \phi^{\mathrm{g}}}^2\lesssim \brho_-^{1+\delta}\brho_+^{1+\delta}\brho_0^{2+\delta} (w^{\mathrm{g}})^{-1} \abs{\Ve w^{\mathrm{g}} \phi^{\mathrm{g}}}^2,
        \end{multline}
        where \cref{scat:eq:Nquad} was used in the second estimate.
    \end{proof}
    We may summarise the relevant part of \cite[Theorem 6.2, Lemma 7.2, Theorem 8.1]{kadar_scattering_2025} as
    \begin{theorem}\label{app:thm:KK25}
        a) Consider $(a_-,a_0,a_+)\in\R^3$ satisfying $a_->a_0>a_+$, $a_+<\frac{d-1}{2}$, $a_->\frac{d-1}{2}$ and $k\geq k_0= \frac{d+2}{2}+5$.
        Let $V,\mathcal{N}$ be admissible potential and nonlinearity of regularity $k+2$ respectively.
        For initial data $\psi_2=0$, $\Vb^k\phi_1\in\A^{0;a_-+\epsilon}(\C)$ and $f\in\A^{k;a_-+1+\epsilon,a_0+2+\epsilon,a_++1+\epsilon}(\D_{\mathrm{g}})$, there exists $\abs{u_1}$ sufficiently large such that \cref{app:eq:semi} has a solution in the region $\D_{\mathrm{g}}\cap\{u<u_1\}$ satisfying $\phi\in\A^{k-\frac{d+2}{2};\vec{a}-}(\D_\mathrm{g})$.

        b) Provided that $a_0>\frac{d-1}{2}$ and $\supp f\subset \{t<0\}$, it furthermore holds that $\phi \in \phi_1+\A^{s';(a_-,a_0,a_0-)}(\D_\mathrm{g})$, where
        \begin{equation}
            \rho_0^{-a_0}\rho_+^{-\frac{d-1}{2}}\left(\Vb\cup \{u^{-1}r^2\partial_v\}\right)^{k'} \phi_1\in L^\infty
        \end{equation}
        for some $k'\sim_{a_0} k$. 
    \end{theorem}

    \subsection{A sketch of a direct proof of \texorpdfstring{\cref{scat:cor:local}}{Proposition 3.1}}\label{sec:app:sketch}
    For the interested reader, we give a rough sketch of the proof of \cref{app:thm:KK25} implemented directly in the local setting of \cref{scat:cor:local} via the coordinate change $\Phi$, focusing on $d\geq3$ for convenience:
    Let us introduce the $L^2$ based norms
    \begin{equation}
        \norm{f}_{\Hb^{k;a_-,a_0,a_+}(\D_{u,v})}^2:=\sum_{\abs{\alpha}\leq k}\int_{D_{u,v}} (w\Vb^{\alpha}f)^2\mu_{\mathrm{b}},\qquad w=\prod_{\bullet}\rho_\bullet^{-a_\bullet}
    \end{equation}
    where $\mu_{\mathrm{b}}=\frac{\dd u}{u}\frac{\dd v}{v}\abs{\slashed{g}}$ is related to the Minkowskian measure as $\dd t\dd x^d=\mu_{\mathrm{b}}\rho_-\rho_0^{d-1}\rho_+$.
	Using weighted Sobolev embedding, we have the following inclusions
	\begin{equation}\label{not:eq:sobolev}
	\Hb^{k+\floor{\frac{d+2}{2}};a_-,a_0,a_+}(\D_{u,v})\subset\A^{k;a_-,a_0,a_+}(\D_{u,v})\subset \Hb^{k;a_--\epsilon,a_0-\epsilon,a_+-\epsilon}(\D_{u,v}),\qquad\forall\epsilon>0.
	\end{equation}

    Let us define the rescaled function $\psi:=\beta\phi$ for $\beta=r^{\frac{d-1}{2}}$ and corresponding twisted derivatives $\tilde\partial_\mu=r^{-\frac{d-1}{2}}\partial_\mu r^{\frac{d-1}{2}}$, $\tilde\partial_\mu^\dag=r^{\frac{d-1}{2}}\partial_\mu r^{-\frac{d-1}{2}}$ satisfying
    \begin{equation}
        \Box_\eta \phi=f\implies \left(\eta^{\mu\nu}\tilde\partial^\dag_\nu\tilde\partial_\mu -r^{-2}c_d\right)\phi=f,
    \end{equation}
    where $c_d r^{-2}=\frac{(d-1)(d-3)}{4r^2}=\beta^{-1}\Box \beta\geq0$.
    
    Let us define the energy momentum tensor
    \begin{equation}
        \tilde{\T}[\phi]_{\mu\nu}:=\tilde\partial_\mu\phi\otimes\tilde\partial_\nu\phi-\frac{\eta_{\mu\nu}}{2}\left(\eta^{\sigma\rho}\tilde\partial_\sigma\phi \tilde\partial_\rho\phi+c_dr^{-2}\phi^2\right).
    \end{equation}
    We recall the following explicit computation from  \cite{kadar_scattering_2025}[Lemma A.2]
    \begin{lemma}\label{app:lemma:divergence}
        Let $V_f=f^u(u,v)\partial_u+f^v(u,v)\partial_v$.
        Then the current $J=V_f\cdot\tilde{\T}[\phi]$ satisfies
        \begin{multline}
            2\mathrm{div}(\tilde{J})=-\beta^{-2}\pv f^u(\pu \psi)^2-\beta^{-2}\pu f^v(\pv \psi)^2+\\\Big(r^{-2}\beta^{-2}V_f(r^{2})-\beta^{-2}(\pu f^u+\pv f^v)\Big)\left(\abs{\sl \psi}^2+c_d(\psi)^2\right)+2\beta^{-1}V_f(\psi)\Box\phi.
        \end{multline}
    \end{lemma}

    For characteristic initial value problems, such as \cref{in:eq:characteristic}, the transversal derivatives on the initial data are determined by transport equations, which lose derivatives.
    We ignore this well understood technical point and focus on deriving estimates  in the case of trivial data and with $f$ is supported away from the two cones
    An appropriate choice of vector fields then yields the following boundedness statement:
    \begin{lemma}
        Fix $a_+<0,a_->1/2$ such that $a_->a_0>a_+$. 
        Let $\beta f\in\Hb^{k;\vec{a}-(1,2,1)}(\D_{-1,1})$ with $\supp f\subset \{u>-1,v>0\}$.
        The unique solution of $\Box \phi=f$ satisfying $\phi|_{\C}=\phi|_{\Cbar_{-1}}=0$ satisfies
        \begin{equation}\label{app:eq:bound}
            \norm{\Ve\psi}_{\Hb^{k;\vec{a}}(\D_{-1,1})}\lesssim_k\norm{\beta f}_{\Hb^{k;\vec{a}-(1,2,1)}(\D_{-1,1})}.
        \end{equation}
    \end{lemma}
        \begin{proof}      
    Let us define for $\bar c\geq0 $
    \begin{nalign}
        J^-[\phi]&=\rho_+^{\bar{c}}v^{-2a_-}\abs{u}^{(2a_--2a_0)}\left( v\pv+\frac{2a_--1}{10}\abs{u}\pu\right)\cdot \tilde\T,\\
        J^+[\phi]&=\rho_-^{-\bar{c}}v^{2a_+-2a_0}\abs{u}^{-2a_+}\left( \abs{u}\pu+v\pv\right)\cdot \tilde\T.
    \end{nalign}
    Let us define $\dot{\Ve}=\Ve\setminus\{1\}$.
    Using \cref{app:lemma:divergence} we obtain the following coercivity estimates (see \cite{kadar_scattering_2025}[Lemmata A.3, A.4]) in $\D^-=\D_{-1,1}\cap\{t<r/2\}$ and $\D^+=\D_{-1,1}\cap\{t>-r/2\}$, respectively:
    \begin{nalign}\label{app:eq:divJpm}
        \mathrm{div}J^-\gtrsim Cv^{-1-2a_-}\abs{u}^{-1+2a_--2a_0}\rho_+^{\bar{c}}\left(1+\bar{c}\rho_-\right)\abs{\beta^{-1}\dot{\Ve}\psi}^2-v^{-1-2a_-}\abs{u}^{-1+2a_--2a_0}\rho_+^{\bar{c}}(\Box\phi)^2u^{2}v^{2},\\
        \mathrm{div}J^+\gtrsim Cv^{-1+2a_+-2a_0}\abs{u}^{-1-2a_+}\rho_-^{-\bar{c}}\left(1+\bar{c}\rho_+\right)\abs{\beta^{-1}\dot{\Ve}\psi}^2-C^{-1}v^{-1+2a_+-2a_0}\abs{u}^{-1-2a_+}\rho_-^{-\bar{c}}(\Box \phi)^2u^{2}v^{2}.
    \end{nalign}
    Note that the integral of the terms on the left in $\D^\pm$ is exactly \cref{app:eq:bound}, e.g. 
    \begin{equation}
        \int v^{-1-2a_++2a_0}\abs{u}^{-1-2a_+}\rho_-^{-\bar{c}}(1+\bar{c}\rho_+)\abs{\beta^{-1}\dot{\Ve}\psi}^2 r^{d-1}\dd u\dd v\abs{\slashed{g}}\sim\norm{\dot{\Ve}\psi}_{\Hb^{0;\vec{a}}(\D^+)}.
    \end{equation}
    Integrating along $\partial_u$ in $\D^-$ and along $\pv$ in $\D^+$ yields the desired control for zeroth order terms.
    Commuting the equation with all the symmetries implies the result for $k\geq0$.
    \end{proof}

    One obtains a small parameter for the nonlinear problem by noting that 
    \begin{equation}
        \norm{f}_{\Hb^{k;(a_-,a_0,a_+)}(\D_{-1,1})}\leq1\implies\norm{f}_{\Hb^{k;(a_-+-\epsilon,a_0-\epsilon,a_+)}(\D_{-1,v})}\lesssim_{k} v^{\epsilon}.
    \end{equation}
    Applying the above lemma together with a bootstrap argument---assuming existence in $\D_{-1,v_1}\cap\{t<T\}$---to the problem $\Box_{\phi}=f+\mathcal{N}[\phi]$ and taking $v_1$ sufficiently small yields \cref{scat:cor:main} a).
    For a linear $V$-term with $\rho_0^{-2}$ behaviour, we can absorb it directly into \cref{app:eq:divJpm} by taking $\bar{c}$ sufficiently large.
    
    We can integrate in the $u$ and $v$ direction to obtain better decay result
    \begin{lemma} \label{app:lemma:uint}
        Fix $a_+<0$ and $a^f_+>-1$.
        For $\phi,f$ supported in $\D^+$ satisfying $\psi\in\Hb^{k;a_0,a_+}(\D^+)$ and $\beta f\in \Hb^{k;a_0-2,a^f_+}(\D^+)$ and $\Box\phi=f$.
        Then $\{1,v\pv\}\psi\in \Hb^{k-2;a_0}(\CH)+\Hb^{k-2;a_0,\min(a_++1,a^f_++1)}(\D^+)$, i.e. there exist $\psi_1(v,x/r)\in \Hb^{k-2;a_0}(\CH)$ and $\psi_2\in \Hb^{k-2;a_0,\min(a_++1,a^f_+-1)}(\D^+)$ such that $\psi=\psi_1+\psi_2$.
    \end{lemma}
    \begin{proof}
        Writing $\partial_u\psi=\int r^{-2}(\mathring{\slashed{\Delta}}+c_d)\psi+\beta f \dd v \in \Hb^{k-2;a_0-1,\min(a_++2,a_+^f)}(\D^+)$ and integrating in $u$ yields the result. See \cite{kadar_scattering_2025}[Corollary 4.1] for details.
    \end{proof}
    Iterating this improvement yields \cref{scat:cor:main} b).

    \subsection{Extending admissibility to treat further nonlinearities} \label{subsec:extendingadmissiblity}
    Note that \cref{scat:cor:local} does not work for nonlinearities of the form $\mathcal{N}=\phi\partial\phi\cdot\partial\phi=\phi\eta^{-1}(\dd\phi,\dd\phi)$, due to the term $\phi\partial_u\phi\partial_v\phi$ appearing in the product.
    In particular, observe that for $a_0>0$ and $a_+<0$, we may only deduce that
    \begin{equation}
        \norm{\partial_u\phi\partial_v\phi}_{\A^{j;a_0-2,2a_+-1}(\D_{v,0}^+)}\lesssim \norm{\Ve\phi}_{\A^{j;a_0,a_+}(\D_{v,0}^+)}^2.
    \end{equation}
    This estimate falls an $a_+$ amount short of satisfying the admissibility in \cref{scat:eq:admiss}.
    With one extra multiplier we can also treat such nonlinear terms.
    Therefore we consider
    \begin{equation}\label{app:eq:characteristic_wm}
        \Box\phi=\phi\partial\phi\cdot\partial\phi+\partial\phi_0\cdot\partial\phi+f,\quad \phi|_{\C}=0,\, \phi|_{\Cbar}=0,
	\end{equation}
    with an inhomogeneity $\beta f\in\Hb^{k;a_0-1,a_0-2,\infty}$ for $a_0>\frac{d-1}{2}$ satisfying $\supp f\subset\{v>0,u>-1\}$ and $\phi_0=\phi_0(x/(t-2r))\in C^\infty_{x/(t-2r)}$.
    Here, $f$ serves as a replacement of data prescribed along $\C,\Cbar$ peeled off as in \cref{scat:peeling}.
    This serves as a scalar analogue of \cref{in:eq:wave_map} outside of symmetry.
    
    Let us observe that for $a_+\in (0,1/2)$ the current (this current is closely related to the $r^p$-current of \cite{dafermos_new_2010})
    \begin{equation}
        J^+_p[\phi]=v^{2a_+-2a_0}\abs{u}^{1-2a_+}\pu\cdot \tilde\T
    \end{equation}
    still yields the coercive estimate for $C$ sufficiently small
    \begin{equation}\label{app:eq:rp}
        \mathrm{div}J^+_p\gtrsim Cv^{-1+2a_+-2a_0}\abs{u}^{-1-2a_+}\abs{\beta^{-1}\{u\pu \}\psi}^2-C^{-1}v^{-1+2a_+-2a_0}\abs{u}^{-1-2a_+}(\Box \phi)^2u^{2}v^{2}.
    \end{equation}

    We apply i) \cref{app:eq:divJpm} with $\vec{a}=(a_0+\epsilon,a_0,a_+)$ with $-a_+=\epsilon\ll 1$ commuted $k$, ii) \cref{app:eq:rp} with $a_+=1/8$ commuted $k$ times, iii) \cref{app:eq:rp} with $a_+=1/4$ commuted $k-4$ times to obtain
    \begin{align}
        \norm{\Ve\psi}_{\Hb^{k;\vec{a}}(\D_{-1,1})}&\lesssim \norm{\beta \Box\phi}_{\Hb^{k;\vec{a}-(1,2,1)}(\D_{-1,1})}\label{app:eq:top}\\
        \norm{\{u\pu\}\psi}_{\Hb^{k;a_0,1/8}(\D^+)}&\lesssim \norm{\beta \Box\phi}_{\Hb^{k;a_0-2,1/8-1}(\D^+)}\label{app:eq:top_p} \\
        \norm{\{u\pu\}\psi}_{\Hb^{k-4;a_0,1/4}(\D^+)}&\lesssim \norm{\beta \Box\phi}_{\Hb^{k-4;a_0-2,1/4-1}(\D^+)}\label{app:eq:low}.
    \end{align}
    From now on, $\vec{a}$ will always refer to $\vec{a}=(a_0+\epsilon,a_0,a_+)$. We use \cref{app:lemma:uint} and \cref{app:eq:top} with a cutoff function localising to $\D^+$ to improve the $\partial_v$ derivative
    \begin{multline}
        \norm{\{1,\pv\}\psi|_{\CH}}_{\Hb^{k-2;a_0}(\CH)}+ \norm{\{1,v\partial_v\}(\psi-\psi|_{\CH})}_{\Hb^{k-2;a_0,1/8}(\D^+)}\lesssim \norm{\beta \Box\phi}_{\Hb^{k;\vec{a}-(1,2,1)}(\D_{-1,1})}\\+\norm{\beta \Box\phi}_{\Hb^{k;a_0-2,1/8-1}(\D^+)}\label{app:eq:rad}. 
    \end{multline}
    Let $X^k[\psi]$ denote the sum of the norms on the left hand side of the above four inequalities.
    Finally, let us estimate the nonlinearities.
    Take $k-4-\frac{d+2}{2}>k/2$.
    We proceed only in the region $\D^+$, the other following strictly simpler computations.
    We bound the right hand side of\cref{app:eq:top,app:eq:top_p,app:eq:rad}
    \begin{align*}
        &\begin{multlined}
            \norm{\beta \phi\partial\phi\cdot\partial\phi}_{\Hb^{k;a_0-2,1/8-1}(\D^+)}\lesssim \norm{u\psi \partial\phi\cdot\partial\phi }_{\Hb^{k;a_0-1,1/8}(\D^+)}\lesssim \norm{u\psi \partial_u\phi\partial_v\phi }_{\Hb^{k;a_0-1,1/8}(\D^+)}\\+\norm{\psi |\mathring{\sl}\phi|^2}_{\Hb^{k;a_0,1/8-1}(\D^+)}
        \end{multlined}\\
        &\begin{multlined}\lesssim \norm{\Vb^{k/2}\{1,v\pv\}\phi}_{L^\infty(\D^+)}^2\norm{u\pu \phi}_{\Hb^{k;a_0,1/8}(\D^+)}
            +\norm{\{1,v\pv\}\phi}^2_{\Hb^{k;0,a_+}(\D^+)}\norm{\rho_+^{-1/4}\Vb^{k/2}u\pu\phi}_{L^\infty(\D^+)}\\+\norm{\Ve\psi}_{\Hb^{k;a_0,a_+}(\D^+)}\norm{\Ve\phi}^2_{\Hb^{k;0,a_+}(\D^+)}
        \end{multlined}\\
        &\begin{multlined}
        \lesssim \left(\norm{\{1,v\partial_v\}\psi|_{\CH}}_{\Hb^{k-2;a_0}(\CH)}+\norm{\{1,v\partial_v\}(\psi-\psi|_{\CH})}_{\Hb^{k-2;a_0,1/10}(\D^+)}\right)^2\norm{u\partial_u\psi}_{\Hb^{k;a_0,1/8}(\D^+)}\\
        +\norm{\{1,v\partial_v\}\psi}_{\Hb^{k;a_0,0-}(\D^+)}^2\norm{u\partial_u\psi}_{\Hb^{k-4;a_0,1/4}(\D^+)}+\norm{\Ve \psi}_{\Hb^{k;a_0,a_+}(\D^+)}^3\lesssim (X^k[\psi])^3,
        \end{multlined}
    \end{align*}
    where $\mathring{\sl}$ denotes $x_i\partial_{x_j}-x_i\partial_{x_i}$ derivatives. 
    Similarly we bound the right hand side of \cref{app:eq:low}
    \begin{multline*}
        \norm{\beta \phi\partial\phi\cdot\partial\phi}_{\Hb^{k-4;a_0-2,1/4-1}(\D^+)}\lesssim \norm{\psi\partial_u\phi\partial_v\phi}_{\Hb^{k-4;a_0-2,1/4-1}(\D^+)}+\norm{\psi|\mathring{\sl}\phi|^2}_{\Hb^{k-4;a_0,1/4-1}(\D^+)}\\
        \lesssim \left(\norm{\{1,v\partial_v\}\psi|_{\CH}}_{\Hb^{k-2;a_0}(\CH)}+\norm{\{1,v\partial_v\}(\psi-\psi|_{\CH})}_{\Hb^{k-2;a_0,1/10}(\D^+)}\right)^2\norm{u\partial_u\psi}_{\Hb^{k-4;a_0,1/4}(\D^+)}\\
        +\norm{\Ve \psi}_{\Hb^{k;a_0,a_+}(\D^+)}^3\lesssim (X^k[\psi])^3.
    \end{multline*}
    Therefore, via a bootstrap argument in $D_{-1,v_1}$ for $v_1$ sufficiently small we obtain:
    \begin{lemma}
        Let $a_0>\frac{d-1}{2}, k>10+d$ and consider an inhomogeneity supported in $\D^-\cap\{u>-1\}$ satisfying $\beta f\in\Hb^{k;a_0-1,a_0-2}(\D^-)$.
        Then the nonlinear wave equation \cref{app:eq:characteristic_wm} has a solution in $\D_{-1,v_1}$ for $v_1\ll 1$ satisfying $X^k[\phi]\lesssim 1$.
    \end{lemma}
    \begin{proof}
        Without the $\phi_0$ term, the result follows by a bootstrap argument and the above estimates.

        For $\phi_0\neq0$, we observe that $\partial_u\phi_0\in\A^{1,-1,0}(\D_{-1,1})\,, \partial_v\phi_0\in \A^{0,-1,1}(\D_{-1,1})$ and $\mathring{\sl} \phi_0\in\A^{0,0,0}(\D_{-1,1})$ thus
        \begin{equation}
            |\partial\phi_0\cdot\partial\phi|\leq r^{-2}\abs{\mathring{\sl}\phi\cdot\mathring{\sl}\phi_0}+\abs{\partial_u\phi_0\partial_v\phi}+\abs{\partial_u\phi\partial_v\phi_0}\lesssim\A^{-1/2,-2,-1/2}\abs{\Ve\phi}.
        \end{equation}
        Using $\bar{c}$ sufficiently large from \cref{app:eq:divJpm} we can bound these linear terms throughout the bootstrap.
    \end{proof}
    
    
    \printbibliography
	
\end{document}